\documentclass[11pt]{article}

\usepackage[margin=1in]{geometry} 
\usepackage{amsmath,amssymb,amsfonts,amsthm}
\usepackage{graphicx}
\usepackage{float}
\usepackage{setspace}
\usepackage{caption}
\usepackage{subcaption}

\usepackage{xcolor}
\usepackage[authoryear, round]{natbib}
\usepackage{hyperref}
\hypersetup{colorlinks,citecolor=blue,linkcolor=blue,breaklinks=true}
 \usepackage{epstopdf}
\usepackage{yhmath} 
\usepackage{makecell}
\usepackage{enumerate}
\usepackage{enumitem}
\usepackage{natbib}

\usepackage{aliascnt}

\allowdisplaybreaks

\newtheorem{theorem}{Theorem}[section]

\newtheorem{lemma}{Lemma}[section]

\newaliascnt{corollary}{theorem}

\aliascntresetthe{corollary}

\newaliascnt{proposition}{theorem}
\newtheorem{proposition}[proposition]{Proposition}
\aliascntresetthe{proposition}

\newtheorem{definition}{Definition}[section]

\newtheorem{remark}{Remark}

\newaliascnt{assumption}{theorem}

\aliascntresetthe{assumption}

\newcommand{\bbm}{\begin{bmatrix}}
	\newcommand{\ebm}{\end{bmatrix}}

\numberwithin{equation}{section}

\begin{document}



\title{EV Charging in Smart Grids: Mean Field Equilibrium and Approximate Non-Cooperative and Cooperative Strategies
\footnotetext{Authors are listed in alphabetical order. \\$^{1}$ School of Mathematics and Statistics, Xidian University, Xi'an, 710126, China. Email: lijunbo@xidian.edu.cn  \\
$^{2}$ School of Mathematical Sciences, University of Science and Technology of China, Hefei, 230026, China. Email: lfc2022@mail.ustc.edu.cn\\
$^{3}$ School of Mathematics and Statistics, Xidian University, Xi'an, 710126, China. Email: wangshihua@xidian.edu.cn}}

\author{
	Lijun Bo $^{1}$
	\and
	Fengcheng Liu $^{2}$
        \and 
    Shihua Wang $^{3}$
}




	
	\date{}
	
	\maketitle

    \begin{abstract}

We study the optimal charging strategies for large-scale electric vehicles in smart grids within a finite-horizon mean field game framework. We first establish the existence and uniqueness of the solution to the consistency condition equation, which characterizes the optimal charging behavior in the mean field limit. Building on this result, we construct approximate optimal charging strategies for a finite population of vehicles in both non-cooperative and cooperative settings. Finally, we provide numerical analyses that illustrate and compare the approximate strategies in the non-cooperative and cooperative games.
    
\vspace{0.1in}


\vspace{0.02in}
    \end{abstract}

\section{Introduction}
As the global energy structure continues to shift towards low-carbonization, the number of electric vehicles (EVs) is growing rapidly. However, the spatiotemporal concentration of large-scale EV charging loads presents a significant challenge to the supply-demand balance of smart grids.  In this context, developing a scientifically robust and highly efficient collaborative optimization strategy framework for EV charging is not only crucial for enhancing the operational resilience of the smart grid but also a core breakthrough for optimizing users' electricity costs and achieving efficient allocation of energy resources {(\cite{Sultan2022} and \cite{Martinez2023})}. Typically, the number of EVs requiring charging in the grid is enormous, leading to the ``curse of dimensionality'' when determining the optimal charging (or storage) strategy for each vehicle.  

In the context of large-scale network structures, mean field games (MFGs) were independently proposed by \cite{Huang-M-C06}  and  \cite{lasry2007mean}, which devoted to the analysis of dynamic systems where a large number of players interact strategically with each other.  By employing the ``mean field'' approximation, MFGs simplify the complex multi-player game problem into an interaction between a single player and the aggregate behavior of the population. This framework provides an effective analytical tool for addressing the interactions among players in large-scale systems. For more comprehensive insights into MFG theory, the reader may refer to \cite{BS16}, \cite{CD18}, \cite{gomes2014mean} and the additional references cited within those works.  
Its application has spanned a wide range of fields including economics and finance ~\citep{lacker2019mean, nuno2018social}, large networks or graphs~\citep{lacker2022case}, operations research~\citep{wang2019mean, wang2022linear}, engineering and machine learning (\cite{ZhouXu2022}), epidemics control (\cite{LaguzetTurinici2015};   \cite{Lee2021}),  smart grid (\cite{GS21}) and more.

Most of the MFG literature in smart grids focuses on two types of problems: Nash equilibria or social optimum. Nash equilibria in mean-field control has been investigated in  \cite{Saldi2019} and \cite{Guo2022}.  The closest work to ours, authored by \cite{CohenZell2024}, examines finite-state, infinite-horizon mean-field games with ergodic costs,  demonstrates that solutions to the mean-field game system yield approximate Nash equilibria in corresponding finite-player games, and establishes a large deviation principle for empirical measures associated with these equilibria. 

Social optimum solutions for MFG are also well documented. For example, \cite{Li2016} examine the relationship between mean-field games and social welfare optimization problems. \cite{Salhab2016} focus on dynamic collective choice problems, and address scenarios where a large number of players cooperatively choose between multiple destinations while influenced by group behavior. \cite{Feng2019} involve a major agent and numerous minor agents and investigates a mixed stochastic LQG social optimization. 



Moreover, versatile framework for modeling price dynamics in markets with numerous interacting agents are applied in MFGs. Notable approaches are listed below: \cite{GS21} introduce a price formation model where numerous small players can store and trade a commodity like electricity; \cite{GGR20} propose a MFG model for price formation of a commodity with production subject to random fluctuations.  The dynamic game models have been extensively utilized to analyze duopolistic competition with sticky prices. Existing literature related to it includes 
\cite{CL04}, \cite{WMB15}, \cite{BGMS21}, \cite{WMB21} and \cite{H21}.


This paper aims to develop a strategic framework for a large market composed of numerous agents, using the methodology of MFG. Within a setting characterized by sticky prices and finite time horizons, we investigate how agents can optimize their behavior to minimize expected losses. In the cooperative game case, our objective is to reduce the average social cost across the market. On the other hand, under the non-cooperative framework, we focus on constructing an approximate Nash equilibrium that captures the decentralized decision-making of individual agents. This dual perspective allows for a comprehensive analysis of collective and competitive behaviors in large-scale smart grid systems.

Given the close resemblance between our model and the models of \cite{Gomes2023}, we now emphasize the distinctions in both model structure and methodological approach. The model in \cite{Gomes2023} determines the mean process of the control variable through a stochastic differential equation. Subsequently, by employing variational methods, they seek a price process such that, when agents take optimal actions to minimize transaction costs, the market clears and supply meets demand. 
Unlike in their work, our price process is endogenously generated within the system and is a stochastic process driven by control variables.
The control variables can be freely adjusted, and the agents' commodities are subject to independent noise disturbances. The more fundamental difference lies in our objective: we aim to construct $\epsilon$-optimal solutions for both cooperative and non-cooperative games in a large system, rather than achieving supply-demand equilibrium. These distinctions lead us to adopt a completely different methodological approach in our analysis-- our methodological approach mainly involves 
 the standard MFE approach, that is, to approximate stochastic processes using fixed functions, rather than employing variational techniques.

Our mathematical contributions can be summarized in three aspects:
First, within a finite-horizon framework, we propose a mean-field interaction cost function, through which we derive approximate solutions for both the cooperative and non-cooperative games, while the literature contains relatively few works that offer both approximate Nash equilibria and approximate social optima. 
Second, for specific matrix equations, we construct the solution induced by specific initial values and study the existence and uniqueness of the mean field through its properties. Most of the existing literature employs the Banach fixed-point theorem or other methods to find equilibrium points. In contrast to these approaches, our method is able to more precisely characterize the state of the mean field.
Third, we apply the Positive Real Lemma to demonstrate that certain parameter restrictions can ensure that the coercivity condition holds, thereby effectively addressing the challenge posed by the positive correlation between the cost function and the product of price and demand term.  To the best of our knowledge, this challenge has not been addressed in the existing literature. 


The rest of the paper is organized as follows: In Section 2, we  introduce the model of the smart grid, which consists of a large number of electric vehicles. In Section 3, we  construct the $\epsilon$-Nash equilibrium charging strategies for the EVs in the non-cooperate games. Afterwards, in section 4, we construct the asymptotic social optimum charging strategies for the cooperate games. In Section 5, we give  numerical simulations for the  charging strategies proposed in both the non-cooperate games and cooperate games. Section 6 concludes.

\section{The Model and the Optimal Strategies}\label{sec:model}

In this section, we propose the sticky price model and the optimal charging strategies of smart grid with numerous agents (e.g. EVs). First, we introduce our model. Let the number of agents in the smart grid be $N$. The agents control their electricity (also referred to as commodities) through controlling the charging rate over a shared common time horizon $[0,T]$. Let $(\Omega, \mathcal{F}, \mathbb{F}, \mathbb{P})$ be a complete filtered probability space with the filtration $\mathbb{F} = (\mathcal{F}_t)_{t\in[0,T]}$ satisfying the usual conditions, which supports the independent $N$-dimension standard Brownian motions $(W_t^1,\cdots,W_t^N)_{t\in[0,T]}$. Denote the electricity held by agent $i$ at time $t$ by $X_t^i$. Let the dynamics of $X^i$ satisfy  
\begin{equation}\label{Xi}
    dX_t^i=v_t^idt+\sigma_i dW_t^i,\quad X_0^i=x_0^i\in\mathbb{R}.
\end{equation}

 Here, the parameter $\sigma_i \ge 0$ measures the volatility of agent $i$'s electricity, and the stochastic noise term represents uncertainty or variability in the smart grid system, such as grid disturbances or physical noise, communication or control delays, or user behavior uncertainty. In addition, we assume that agent $i$ is equipped with an initial electricity $x_0^i\in \mathbb{R}$, 
 As  in \eqref{Xi}, the control variable for agent $i$, 
  also referred to as trading rate (or charging rate) is denoted by $v^i$. To introduce the admissible strategy set of $v^i$, we assign every stochastic control 
$w$ the norm $\left\| w \right \|=\sqrt{\mathbb{E}\left[\int_0^T w_t^2  dt\right]}$. Then, $v^i$ should belong to the admissible strategy set
 \begin{align*}
    \mathcal{A}_i=\Big\{v^i\ |\  &v^i \text{ is a  real-value process progressively measurable w.r.t } \mathbb{F},\  \|v^i\|< \infty,\\
    & \text{and there exists a unique strong solution to \eqref{Xi}}\Big\} ,
\end{align*}

We then introduce the evolution of the price process $P$. We consider a dynamic oligopoly setting, where prices do not instantly adjust to changes in market conditions, such as shifts in supply or demand. Instead, its evolution is affected by its current price and average trading rate. A natural way of modeling the sticky price in dynamic
oligopoly is (\cite{CL04} and \cite{WMB15})
\begin{equation}
    d P_t=\alpha(\beta-Q_t-P_t)dt,\quad P_0=p_0\label{price process},
\end{equation}
where $\alpha, \beta >0$ are constant market factors, and the initial price is $p_0$. The parameter $\alpha$ measures the sensibility of $P$ to current situations, and $Q_t:=\frac{1}{N}\sum_{i=1}^N v_t^i$ is the average trading rate at time $t$.
The sticky price model  well captures the persistent effects of the past price conditions on the current prices, reflecting realistic frictions in electricity markets.

In the following, by referring to existing literature, we propose a reasonable cost function for each agent in the game. Inspired by \cite{Gomes2023}, we let the expected overall cost of agent $i$ be of the following form
\begin{equation} \label{LQObject}
J_i(v^i, \boldsymbol{v}^{-i}) =\mathbb{E}\left[\int_0^T\left(L(X_t^i,v_t^i)+P_tv_t^i\right)dt+\Psi(X_T^i)\right],
\end{equation}
where $\boldsymbol{v}^{-i}:=\{v^1,\cdots,v^{i-1},v^{i+1},\cdots, v^N\}$. 
We assume the running cost and the terminal cost to be
\begin{equation*}
    L(x,v)=\frac{\eta}{2}(x-\kappa)^2+\frac{c}{2}v^2\quad\mathrm{and}\quad\Psi(x)=\frac{\gamma}{2}(x-\zeta)^2
\end{equation*}
respectively. The parameter $\zeta$ corresponds to the preferred final storage, and $\kappa$ is the preferred instantaneous storage. We make the assumption that $ \eta>0,\  \gamma\ge 0,\  c>0,\  \kappa\in \mathbb{R},\  \text{and}\ \zeta\in \mathbb{R} $ are constants. Compared with the conventional LQG framework, this cost function includes an additional $vP$ term. The introduction of this term weakens the coercivity condition, and increases the difficulty of estimating $\epsilon_N$ and deriving the optimal strategies.

Finally, building on the previously proposed cost function, we introduce an objective function for the system. We consider two different game cases.
In one case (namely, the non-cooperate game case), each agent controls its commodity independently, and does not take into account how its actions might affect others' expected cost. In other words, agent $i$ only aims to minimize its expected value function $J_i$, and does not consider how $J_k (\forall1\le k\le N, k\neq i)$ changes according to its actions. Under this setting, it is natural for us to find an $\epsilon$-Nash equilibrium for the system.

A Nash equilibrium is a key concept in game theory that represents a stable state in a strategic interaction where no player can improve their outcome by unilaterally changing their strategies. It occurs when each player's strategy is optimal given the strategies chosen by all other players. However, as $N$ enlarges, it becomes rather challenging to find the exact Nash equilibrium solution, so we aim to find a set of strategies that approaches the best outcome as $N$ enlarges, referred to as a set of $\epsilon$-Nash equilibrium strategies. An $\epsilon$-Nash equilibrium characterizes a situation where each player's chosen strategy is almost the best response to the strategies of the others, with the difference between the expected personal cost of this strategy and that under the optimal strategy not exceeding a positive number $\epsilon_N$, which vanishes as $N$ goes to $\infty$. We introduce the precise definition of an $\epsilon$-Nash equilibrium:
\begin{definition}
    A set of admissive strategies $\hat{\boldsymbol{v}}=\{\hat{v}^1,\cdots,\hat{v}^N\} \in \Pi_{i=1}^N\mathcal{A}^i$ is an $\epsilon$-Nash equilibrium if
\begin{equation}
    J_i(\hat{v}^i,\hat{\boldsymbol{v}}^{-i})-\epsilon_N\le\inf_{v^i \in \mathcal{A}_i} J(v^i,\hat{\boldsymbol{v}}^{-i})\le J_i(\hat{v}^i,\hat{\boldsymbol{v}}^{-i}), \quad \forall i= 1,\ldots, N \label{Nash Obj},
\end{equation}
where $\epsilon_N$ goes to 0 as $N\to \infty$.
\end{definition}

Under the second case (namely, the cooperate game case), instead of controlling the electricity independently, the agents cooperate with each other, and aim to achieve the best collective outcome. In this article, the collective outcome represents the average value of $J_i(v^i, \boldsymbol{v}^{-i})$. It reads
\begin{equation*}
    J_{soc}(v^1,v^2,\cdots,v^N)=\frac{1}{N}\sum_{i=1}^N\mathbb{E}\left[\int_0^T\left(L(X_t^i,v_t^i)+P_tv_t^i\right)dt+\Psi(X_T^i)\right].
\end{equation*}
In large-scale systems, finding the exact social optimum becomes complex, so we focus on finding a set of strategies that approaches the best collective outcome as the number of agents tends to infinity. More specifically, we introduce the following definition:
\begin{definition}
    We say that a set of admissive strategies $\{\check{v}^i\}_{i=1}^N$ is a set of $\epsilon$-optimal social strategies if
\begin{equation}
    |J_{soc}(\check{v}^1,\check{v}^2,\cdots,\check{v}^N)-\inf_{v^i\in \mathcal{A}_i,1\le i\le N} J_{soc}(v^1,v^2,\cdots,v^N)|\le \epsilon_N, \label{Social Obj}
\end{equation}
where $\epsilon_N$  goes to 0 as $N\to \infty$.
\end{definition}
Here, the restriction for $\epsilon_N$  reflects that the expected social cost under the proposed strategies converges to that of the social optimal strategies as $N$ goes to $\infty$.


It is worth emphasizing that asymptotic social optimum strategies may not be $\epsilon$-Nash equilibrium strategies. Under the asymptotic social optimum strategies, an agent may significantly reduce its own expected cost by only changing its own strategy, while causing even more cost for others, and comprehensively, give rise to the overall expected cost.


\section{Non-cooperate Game}
In this section, we aim to find an $\epsilon$-Nash equilibrium for this system. We will use a mean-field approximation approach, in which we construct $C[0,T] $ functions to approximate certain processes. Our approach can be outlined in the following steps.
\begin{enumerate}
  \item  The first step in constructing a mean-field approximation is to approximate $Q$ with a continuous  function $\bar{Q}=(\bar{Q}(t))_{t\in[0,T]}$. It follows from \eqref{price process} that the approximate price function $(\bar{P}(t))_{t\in[0,T]}$ satisfies 
  \begin{equation}\label{approximate price process}
      d \bar{P}(t)=\alpha(\beta-\bar{Q}(t)-\bar{P}(t))dt,\quad
      \bar{P}(0)=p_0.
  \end{equation}

 \item Replace $P_t$ by $\bar{P}(t)$ in objective function $J_i$ to obtain an auxiliary objective function
 \begin{equation}\label{barJ}
    \bar{J}_i(v^i;\bar{P})=\mathbb{E}\left[\int_0^T (L(X_t^i,v_t^i)+\bar{P}(t) v_t^i)dt
+\Psi(X_T^i)\right],
\end{equation}
subject to state process $(X_t)_{t\in[0,T]}$, which satisfies \eqref{Xi}.
\item For each agent $i$, solve the auxiliary control problem \eqref{barJ} to find the optimal control $\hat{v}^i$. We calculate the average of the corresponding  optimal controls and remove the noise term to obtain a function which measures the expectation of the average of the optimal controls. By requiring that this function equal $\bar{Q}$,  we determine the exact form of $\bar{Q}$.
\item Prove that under certain assumptions, the set of controls derived from $\bar{Q}$ is indeed an $\epsilon$-Nash equilibrium.
\end{enumerate}
\begin{remark}
The rationale for proposing the first two steps is primarily based on the following reasons. Firstly, when the population enlarges to $\infty$, the representative agent has no influence on $Q$, and therefore has no influence on $P$, as but one agent amid a continuum. From its perspective, the process $P$ should be evolved in its cost as a deterministic function. Secondly, the optimization problem for  (\ref{LQObject}) becomes much easier to solve once we treat $P$ as a fixed continuous function, instead of a stochastic process. For each  
$\bar{Q}$, a unique process $\bar{P}$ is determined by \eqref{approximate price process}. Each agent determines its best response to minimize \eqref{barJ}. By calculating the expectation of the average of the optimal controls, we derive a function determined by $\bar{Q}$. If the expectation of this function is equal to $\bar{Q}$, a closed-loop is formed. Within this loop, the strategy adopted by each agent yields outcomes that closely approximate those of the actual optimal strategy as the number of agents increases.
\end{remark}

For the given function $\bar{Q}\in C[0,T]$, we derive the  approximate price process $\bar{P}$  by \eqref{approximate price process}. We propose the value function of the auxiliary control problem \eqref{barJ}, 
\begin{equation*}
    K_i(t,x)=\inf_{v \in \mathcal{B}_i}\mathbb{E}\left[\int_t^T (L(X_s,v_s)+\bar{P}(s) v_s)ds
+\Psi(X_T)\right].
\end{equation*}
where $t$ represents the starting time,  $X$ satisfies \begin{equation}\label{Xi2}
    d X_s =v_sds+\sigma_i dW_s,\quad \forall t\le s\le T;\quad X_t=x,
\end{equation}
and the admissible set for $v$ is denoted by 
\begin{align*}
    \mathcal{B}_i=\Big\{v\ |\  &v \text{ is a  real-value process progressively measurable w.r.t } (\mathbb{F}_s)_{t\le s\le T},\  \mathbb{E}\left[\int_t^T v^2_s ds\right]< \infty,\\
    & \text{and there exists a unique strong solution to \eqref{Xi2}}\Big\} ,
\end{align*}
By dynamic programming principle,  the HJB equation satisfies 
\begin{equation}\label{HJB eq for Nash}
   \frac{\partial {K}_i(t,x)}{\partial t} +\inf_{v^i \in \mathcal{B}_i} \left\{v_t^i \frac{\partial {K}_i(t,x)}{\partial x} +\frac{c}{2} ({v_t^i})^2+v_t^i\bar{P}(t)\right\}+\frac{\sigma_i^2}{2}\frac{\partial^2 {K}_i(t,x)}{\partial x^2}+\frac{\eta}{2}  (x-\kappa)^2=0,
\end{equation}
along with the terminal condition $K_i(T,x)=\frac{\gamma}{2}(x-\zeta)^2$. We make the ansatz ${K}_i(t,x)=a_i(t)x^2+ B_i(t)x+F_i(t)$ yields. Furthermore, the optimal feedback control achieving
the minimum is given by
\begin{equation}\hat{v}_t^i=-\frac{\bar{P}(t)+2a_i(t)x+B_i(t)}{c},\quad \forall t\in [0,T].
   \label{Optimal strategy for Nash} 
\end{equation}
By substituting \eqref{Optimal strategy for Nash} into \eqref{HJB eq for Nash}, we obtain that
\begin{align}
    &\qquad\qquad\qquad a'_i(t) + \frac{\eta}{2} - \frac{(2a_i(t))^2}{2c} = 0, & & a_i(T) = \frac{\gamma}{2}, & \label{Nash HJB eq1} \\
    &\qquad\qquad\qquad B'_i(t) - \eta\kappa - \frac{2a_i(t)(\bar{P}(t) + B_i(t))}{c} = 0, & & B_i(T) = -\gamma\zeta, & \label{Nash HJB eq2} \\
    &\qquad \qquad\qquad  F'_i(t) - \frac{(\bar{P}(t) + B_i(t))^2}{2c} + \frac{\eta\kappa^2}{2} + \sigma_i^2a(t) = 0, & & F_i(T) = \frac{\gamma\zeta^2}{2}. & \label{Nash HJB eq3}
\end{align}

We will subsequently verify that \eqref{Nash HJB eq1}–\eqref{Nash HJB eq3} indeed admit a solution. To this end, we introduce \autoref{a(t) exists theorem}.
\vskip 3mm
\begin{proposition}\label{a(t) exists theorem}
  The solution to (\ref{Nash HJB eq1}), (\ref{Nash HJB eq2}) and (\ref{Nash HJB eq3}) exists and is unique.
\end{proposition}
\begin{remark}
   Throughout this paper, unless stated otherwise, we seek classical solutions to the equations under consideration. Since the solution to  \eqref{Nash HJB eq1} and \eqref{Nash HJB eq2} is uniform over $i$, so in the following, we abbreviate $a_i(t), B_i(t)$ as $a(t),B(t)$, respectively.
\end{remark}
 The proofs of the Propositions in this paper are collected to Appendix \ref{sec:appendix2}. 

Next, we present the verification theorem for the best-response control that minimizes \eqref{barJ}.

\begin{theorem}\label{thm:veri}
    (Verification theorem) We introduce the feedback control function as follows:
    \begin{equation*}
        \hat{v}_t^i=-\frac{\bar{P}(t)+2a(t){X}_t^i+B(t)}{c},\quad \forall t\in [0,T],
    \end{equation*}
  where $a(t), B(t)$ are given by \eqref{Nash HJB eq1} and \eqref{Nash HJB eq2}.  Then, the SDE 
    \begin{equation*}
    d{X}_t^i=\hat{v}_t^idt+\sigma_idW_t^i,\quad {X}_0^i=x_0^i
    \end{equation*}
    admits a unique strong solution, denoted by $\hat{X}_t^i$, and $\hat{v}^i$
 is an optimal Markovian control for minimizing \eqref{barJ}. \end{theorem}
\begin{proof}
$\hat v:(t,x)\to -\frac{\bar{P}(t)+2a(t)x+B(t)}{c}$ is   Lipschitz-continuous w.r.t $(t,x)\in [0,T]\times \mathbb{R}$, implying that \eqref{Xi} admits a unique strong solution (Theorem 8.3 in \cite{Gall}).  
Since  $a(t)$, $B(t)$ and $F_i(t)$ are bounded on $[0,T]$, the function $K_i(t,x)$ satisfies the quadric growth condition; that is, there exists $C>0$ such that $|K_i(t,x)|\le C(1+x^2),\ \forall(t,x)\in [0,T]\times \mathbb{R}$. Additionally, the terminal condition $K_i(T,\cdot)=\Psi(\cdot)$ is fulfilled. Under these  conditions on $K_i$, together with 
\begin{align*}
     &\frac{\partial {K}_i(t,x)}{\partial t} +\inf_{v^i \in \mathcal{A}_i} \left\{v_t^i \frac{\partial {K}_i(t,x)}{\partial x} +\frac{c}{2} ({v_t^i})^2+v_t^i\bar{P}(t)\right\}+\frac{\sigma_i^2}{2}\frac{\partial^2 {K}_i(t,x)}{\partial x^2}+\frac{\eta}{2}  (x-\kappa)^2\\
     &= \frac{\partial {K}_i(t,x)}{\partial t}  +\hat{v}_t^i \frac{\partial {K}_i(t,x)}{\partial x} +\frac{c}{2} ({\hat{v}_t^i})^2+\hat{v}_t^i\bar{P}(t)+\frac{\sigma_i^2}{2}\frac{\partial^2 {K}_i(t,x)}{\partial x^2}+\frac{\eta}{2}  (x-\kappa)^2=0,
\end{align*}
we obtain the  result  by applying the arguments in  Theorem 3.5.2, \cite{Huyen}.
\end{proof}

We now approximate the average commodity process by a deterministic function. The rationale behind this approximation is as follows. Let $\overline{X}_t$ denote the average of $X_t^i$.  
Then the average optimal control takes the form $$-\frac{2a(t) \overline{X}_t+{B}(t)+\bar{P}(t)}{c},$$ which implies that the mean trading process evolves as follows
\begin{equation*}
    d \overline{X}_t=-\frac{2a(t) \overline{X}_t+{B}(t)+\bar{P}(t)}{c}dt+\frac{1}{N}\sum_{i=1}^N \sigma_i dW_t^i, \quad\overline{X}_0=\overline{x}_0^N.
\end{equation*}
Note that $\frac{1}{N}\sum_{i=1}^N \sigma_i dW_t^i$ represents a Brownian motion with magnitude $\frac{1}{N}\sqrt{\sum_{i=1}^N\sigma_i^2}$. This magnitude clearly vanishes to zero as $N\to \infty$ provided that  $\sigma_i\ (\forall i=1,2,\ldots,N)$ are  bounded by a  constant independent of $N$. Hence, under the standard mean-field approximation, we replace the average stochastic process by a deterministic function $\bar{x}(t)$, whose dynamics satisfy
\begin{equation*}
    d\bar{x}(t)=-\frac{2a(t)\bar{x}(t)+B(t)+\bar{P}(t)}{c}dt,\quad \bar{x}(0)=\bar{x}_0^N.
\end{equation*}
By consistency condition, the expectation of the induced average trading process must satisfy
\begin{equation*}\label{eq:consist}
    -\frac{2a(t)\bar{x}(t)+B(t)+\bar{P}(t)}{c} = \bar{Q}_t , \quad  \forall t\in[0,T].   
\end{equation*}
Using the relationship between $\bar{P}(t), B(t)$ and $\bar{Q}(t)$, this requirement is equivalent to the following system:
\begin{equation*}\label{consist-sym}
\left\{
\begin{aligned}
&d{\bar{P}(t)}=\alpha(\beta-\bar{P}(t)-\bar{Q}(t))dt,\quad\bar{P}(0)=p_0, \\
&B'(t)-\eta\kappa -\frac{2a(t)(\bar{P}(t)+B(t))}{c}=0,\quad B(T)=-\gamma\zeta,\\
&d\bar{x}(t)=-\frac{\bar{P}(t)+2a(t)\bar{x}(t)+B(t)}{c}dt, \quad\bar{x}(0)=\bar{x}_0^N,\\
&\bar{Q}(t)=-\frac{\bar{P}(t)+2a(t)\bar{x}(t)+B(t)}{c}.
\end{aligned}
\right.
\end{equation*}
By introducing matrix notation, we obtain an alternative form, which requires the following function system
\begin{equation}
    \frac{d}{dt}
\begin{pmatrix}
B(t)\\
\bar{x}(t)\\

\bar{P}(t)
\end{pmatrix}
=
\begin{pmatrix}
    \frac{2a(t)}{c}&0&\frac{2a(t)}{c}\\
    -\frac{1}{c} &-\frac{2a(t)}{c}&-\frac{1}{c}\\
    \frac{\alpha}{c}&\frac{2\alpha a(t)}{c}&-\alpha+\frac{\alpha}{c}

\end{pmatrix}
\begin{pmatrix}
B(t)\\
\bar{x}(t)\\

\bar{P}(t)
\end{pmatrix}
+
\begin{pmatrix}
    \eta \kappa\\
    0\\
    \alpha \beta

\end{pmatrix},
\label{mean field process evolution}
\end{equation}
to hold simultaneously with the following initial or terminal conditions
\begin{equation}\label{mean field process evolution:boundary value}
    \begin{pmatrix}
B(T),\ 
\bar{x}(0),\ 
\bar{P}(0)
\end{pmatrix}
^\top=
\begin{pmatrix}
-\zeta\gamma,\ 
\bar{x}_0^N,\ 
p_0
\end{pmatrix}
^\top.
\end{equation}
The existence of solutions to (\ref{mean field process evolution}) and (\ref{mean field process evolution:boundary value}) is essential, as it determines whether the consistency requirement in Step 3 can be fulfilled. The main challenge here is:  the boundary conditions are imposed at different time points.

To tackle this issue, we first study solutions with initial conditions specified at the same time point, and establish their properties in \autoref{propo for existence} and \autoref{propo for existence 2}. These results serve as preparatory steps for \autoref{A_2 imply existence}, where we prove that, under \textbf{Assumption} 
$\boldsymbol{(A_1)}$ stated below, the coupled system \eqref{mean field process evolution}–\eqref{mean field process evolution:boundary value} admits a unique solution.


\begin{proposition}\label{propo for existence}
  For $\forall b_0\in\mathbb{R}$, the initial condition $(B(0),\ \bar{x}(0),\ \bar{P}(0))^{\top}= (b_0,\ 
\bar{x}_0^N,\ 
p_0)^\top$
conducts a unique solution $(\phi_{b_0}(t))_{t\in[0,T]}=(B_{b_0}(t),\ \bar{x}_{b_0}(t), \ \bar{P}_{b_0}(t))^\top_{t\in[0,T]}$  that satisfies (\ref{mean field process evolution}).
\end{proposition}
 \autoref{propo for existence} shows that once the conditions $\bar{x}(0)=\bar{x}_0^N$ and $\bar{P}(0)=p_0$ are given, the solution to \eqref{mean field process evolution} is uniquely determined by the choice of $B(0)=b_0$ . Therefore, our remaining task is to adjust  the value of $b_0$ so that $B(T)=-\gamma\zeta$ holds.

To this end, we construct the following auxiliary function system, 
\begin{equation}\label{A_1 process}
     \frac{d}{dt}
\begin{pmatrix}
B_1(t)\\
\bar{x}_1(t)\\

\bar{P}_1(t)
\end{pmatrix}
=
\begin{pmatrix}
    \frac{2a(t)}{c}&0&\frac{2a(t)}{c}\\
    -\frac{1}{c} &-\frac{2a(t)}{c}&-\frac{1}{c}\\
    \frac{\alpha}{c}&\frac{2\alpha a(t)}{c}&-\alpha+\frac{\alpha}{c}
\end{pmatrix}
\begin{pmatrix}
B_1(t)\\
\bar{x}_1(t)\\

\bar{P}_1(t)
\end{pmatrix}
,\forall t\in [0,T];
\begin{pmatrix}
B_1(0)\\
\bar{x}_1(0)\\

\bar{P}_1(0)
\end{pmatrix}=
\begin{pmatrix}
1\\
0\\
0
\end{pmatrix}.
\end{equation}

The existence and uniqueness of \eqref{A_1 process} are established below in \autoref{propo for existence 2}. 
\begin{proposition}\label{propo for existence 2}
    There exists a unique solution to (\ref{A_1 process}).
\end{proposition}

We now introduce \textbf{Assumption} $\boldsymbol{(A_1)}$ and show that, under this assumption, the coupled system \eqref{mean field process evolution}–\eqref{mean field process evolution:boundary value} indeed admits a solution.

\noindent\textbf{Assumption} $\boldsymbol{(A_1):}$  The solution to  \eqref{A_1 process} satisfies $B_1(T)\neq 0$.

\begin{proposition}\label{A_2 imply existence}
     If  \textbf{Assumption} $\boldsymbol{(A_1)}$  holds, then (\ref{mean field process evolution}), combined with (\ref{mean field process evolution:boundary value}), gives a unique solution. 
\end{proposition}

We further impose an assumption that constrains the magnitudes of the elements in the sets  $\{\sigma_i\}_{1\le i\le N} $ and $\{x_0^i\}_{1\le i\le N}$. 

\vspace{2mm}
\noindent\textbf{Assumption}
    $\boldsymbol{(A_2):}$ The sets
$\{x_0^i\}_{i=1}^N$ and 
$\{\sigma_i\}_{i=1}^N$ are respectively defined on two fixed compact sets that are independent of $N$.

Now, we propose the main theorem for this section.
\begin{theorem}\label{Nash main theorem}
  Assume that \textbf{Assumption} $\boldsymbol{(A_1)}$ and $\boldsymbol{(A_2)}$ hold, and the unique solution to (\ref{mean field process evolution}) and (\ref{mean field process evolution:boundary value}) is given by $(B(t),  \
\bar{x}(t), \  
\bar{P}(t))^\top$,
then the following system
\begin{equation}\label{Nash process}
\left\{
\begin{aligned}
 &d\hat{P}_t=\alpha (\beta -\hat{Q}_t-\hat{P}_t)dt, \quad  \hat{P}_0=p_0, \\
 & \hat{v}_t^i=-\frac{\bar{P}(t)+2a(t)\hat{X}_t^i+B(t)}{c}, \\ 
 & d \hat{X}_t^i=\hat{v}_t^i dt+\sigma_i dW_t^i, \quad \hat{X}_0^i=x_0^i, \\
 & \hat{Q}_t=\frac{1}{N}\sum_{i=1}^{N} \hat{v}_t^i.
\end{aligned}
\right.
\end{equation}
is an $\epsilon$-Nash equilibrium. Moreover, specific calculations show that under the proposed strategies $\{\hat v_t^i\}_{i=1}^N$, $\epsilon_N=O\Big(\frac{1}{\sqrt{N}}\Big)$ in \eqref{Nash Obj}.
\end{theorem}

To prove \autoref{Nash main theorem}, we will need three auxiliary lemmas. As a first step, in order to derive uniform bounds for the norms of the state processes under the optimal control and for the quantities $J_i(\hat v^i,\hat{\boldsymbol{ v}}^{-i}),\ \forall i=1,2,\ldots,N$, we introduce \autoref{Proposed control bounded lemma for Nash}.

\begin{lemma}\label{Proposed control bounded lemma for Nash}
  There exists a positive constant $K'$, independent of $N$ and $i$, such that
   \begin{equation*}
   \left\{
       \begin{aligned}
       & \mathbb{E}\left[\sup_{0\leq t\leq T }|\hat{X}_t^i|^2\right]\le K',\  \quad   \mathbb{E}\left[\int_0^T (\hat{v}_t^i)^2 dt\right] \le T \mathbb{E}\left[\sup_{0\le t\le T }(\hat{v}_t^i)^2 dt\right]\le K',\quad i=1,2,\ldots,N,  \\
        &  \mathbb{E}\left[\sup_{0\le t\le T }|\overline{\hat{X_t}}|^2\right]\le K',\ \quad  \mathbb{E}\left[\sup_{0\le t\le T }\left|\frac{1}{N}\sum_{k\neq i} \hat{v}_t^k\right|^2\right]\le K'\ ,\quad i=1,2,\ldots,N, \\
       & \mathbb{E}\left[\sup_{0\le t\le T }|{\hat{P}_t}|^2\right]\le K',\  \quad \mathbb{E}\left[\int_0^T \hat{Q}_t^2 dt\right] \le T\mathbb{E}\left[\sup_{0\le t\le T }|\hat{Q}_t|^2\right]\le K', \\
       &  |J_i(\hat{v}^i,\hat{\boldsymbol{v}}^{-i})|\le K',\quad  i=1,2,\ldots,N, 
       \end{aligned}
       \right.
   \end{equation*}
where $\overline{\hat{X}}_t$ is the average of $\hat{X}_t^i$, from $i=1$ to $N$.
\end{lemma}
The proofs of all lemmas  have been left to \autoref{sec:appendix}.

To show that the proposed strategy profile constitutes an 
$\epsilon$-Nash equilibrium, we demonstrate that for any agent $i$, its expected cost $J_i(\hat{v}^i,\hat{\boldsymbol{v}}^{-i})$ cannot be significantly reduced by a unilateral deviation. Specifically, for agent $i \ (\forall i = 1, \ldots, N)$, consider an arbitrary admissible control $v^i\in \mathcal{A}_i$, while other agents continue using  $\hat{\boldsymbol{v}}^{-i}$. Under such a deviation, the commodity process of agent $i$ evolve as follows:
\begin{equation*}
    d X_t^i= v_t^idt+\sigma_i dW_t^i,\quad X_0^i=x_0^i,
\end{equation*}
and the price dynamics become
\begin{equation*}
   d{P}_t=\alpha \bigg(\beta -\frac{v_t^i}{N}-\frac{1}{N} \sum_{k\neq i} \hat{v}_t^k-{P}_t\bigg)dt, \quad  {P}_0=p_0.
\end{equation*}
We aim to compare $J_i(\hat{v}^i,\hat{\boldsymbol{v}}^{-i})$ with $J_i(v^i,\hat{\boldsymbol{v}}^{-i})$ and establish \eqref{Nash Obj} with $\epsilon_N=O(\frac{1}{\sqrt{N}})$, or equivalently, 
\begin{equation}\label{equiobj}
    J_i(v^i, \hat{\boldsymbol{v}}^{-i})- J_i(\hat{v}^i, \hat{\boldsymbol{v}}^{-i})\ge O\bigg(\frac{1}{\sqrt{N}}\bigg),\quad\forall \ v^i \in \mathcal{A}_i.
\end{equation}

To establish \eqref{equiobj}, we require a coercivity condition, whose precise formulation is given in \autoref{Nash coercive lemma}. Before presenting it, we first introduce \autoref{PRL for Nash}, which will streamline the proof of the coercivity result.
\begin{lemma}\label{PRL for Nash}
    For any $N\in \mathbb{N}^+$ and $v\in \mathcal{A}_i$, we construct the process
\begin{equation*} \label{lemma 2.2.2 statement}
    \frac{d}{dt}
\begin{pmatrix}
    X_t^*\\
    P_t^*
\end{pmatrix}
=
\begin{pmatrix}
    0 &0\\
    0 &{-\alpha}
\end{pmatrix}
\begin{pmatrix}
    X_t^*\\
    P_t^*
\end{pmatrix}
+
\begin{pmatrix}
    1\\
    -\frac{\alpha}{N}
\end{pmatrix}
v_t,\ \quad   X_0^*=0,\  P_0^*=0.
\end{equation*}

Then, for any given $ \epsilon_1^*>0$,  $\epsilon_2^* >0$ and $N\ge\frac{\sqrt{2}+1}{2{\epsilon_2}^*}$, we have 
\begin{equation}\label{leq:XPstar}
   \mathbb{E}\left[ \int_0^T v_t (\epsilon_1^{*}X_t^*+P_t^*+{\epsilon}_2^* v_t)dt\right]\ge 0.
\end{equation}
\end{lemma}


 \begin{lemma}\label{Nash coercive lemma}
     There exists a constant $C_0$, which is independent of $N$ and $i$, such that the coercive condition
\begin{equation*}
    J_i(v^i,\hat{\boldsymbol{v}}^{-i})\ge \mathbb{E}\left[ \int_0^T \Big(\frac{\eta}{4} (X_t^i-\kappa)^2+\frac{c}{4}(v_t^i)^2\Big)dt\right]+C_0
\end{equation*}
is satisfied for any $N>\frac{4(\sqrt{2}+1)}{c}$.
 \end{lemma}
We now return to our proof of  \autoref{Nash main theorem}.
\begin{proof}[ Proof of  \autoref{Nash main theorem}]
According to  \autoref{Proposed control bounded lemma for Nash}, we get
\begin{equation*}
    |J_i(\hat{v}^i,\hat{\boldsymbol{v}}^{-i})| \le K'.
\end{equation*}
 For all $ N>\frac{4(\sqrt{2}+1)}{c} $, if $  \frac{c}{4} \|v^i\|^2+C_0\ge K'$, then by   \autoref{Nash coercive lemma}, we immediately obtain that
\begin{equation*}
    J_i(v^i, \hat{\boldsymbol{v}}^{-i})\ge J_i(\hat{v}^i, \hat{\boldsymbol{v}}^{-i}),
\end{equation*}
which yields (\ref{equiobj}).
Otherwise, $\|v^i\|^2$ is bounded by $\frac{4(K'-C_0)}{c}$. Hence, it suffices to establish \eqref{equiobj} in the case  where $\|v^i\|$ is bounded by an $O(1)$ constant.

As noted earlier, $\bar{P}(t)$ is constructed for approximating $P_t$ and $\hat{P}_t$. This motivates us to express the difference of the expected cost as
\begin{align*}
  & J_i(v^i, \hat{\boldsymbol{v}}^{-i})- J_i(\hat{v}^i, \hat{\boldsymbol{v}}^{-i}) \nonumber \\
 & =\mathbb{E}\left[\int_0^T (L(X_t^i,v_t^i)+\bar{P}(t) v_t^i)dt
+\Psi(X_T^i)\right]-\mathbb{E}\left[\int_0^T  (L(\hat{X}_t^i,\hat{v}_t^i)+\bar{P}(t) \hat{v}_t^i)dt
+\Psi(\hat{X}_T^i)\right]  \nonumber \\
&\quad+\mathbb{E}\Big[\int_0^T v_t^i (P_t-\bar{P}(t))dt\Big] -\mathbb{E}\Big[\int_0^T \hat{v}_t^i (\hat {P}_t-\bar{P}(t))dt\Big]  \nonumber \\
&:=I_1+I_2-I_3,
\end{align*}
where
\begin{align*}
I_1&=\mathbb{E}\left[\int_0^T  (L({X}_t^i,{v}_t^i)+\bar{P}(t) v_t^i)dt+\Psi({X}_T^i) -\left(\int_0^T  (L(\hat{X}_t^i,\hat{v}_t^i)+\bar{P}(t) \hat{v}_t^i)dt
+\Psi(\hat{X}_T^i)\right)\right],\\
I_2&=\mathbb{E}\left[\int_0^T v_t^i (P_t-\bar{P}(t))dt\right],\\
I_3&=\mathbb{E}\left[\int_0^T \hat{v}_t^i (\hat {P}_t-\bar{P}(t))dt\right].
\end{align*}

    
Thus, to finish the proof, it suffices to show 
\begin{align}\label{ooo}
    I_1\ge 0, \quad 
   I_2=O\Big(\frac{1}{\sqrt{N}}\Big), \quad  I_3=O\Big(\frac{1}{\sqrt{N}}\Big).
\end{align}

Since we have already established that $\hat{v}_t^i$ is optimal for minimizing
\begin{equation*}
    \mathbb{E}\left[\int_0^T  (L(X_t^i,v_t^i)+\bar{P}(t) v_t^i)dt+\Psi(X_T^i)\right],
\end{equation*}
it follows directly that $ I_1\ge 0$.

We now address the remaining term in \eqref{ooo}.
Direct calculation yields
\begin{align*}
& d\overline{\hat{X}}_t=-\frac{\bar{P}(t)+2a(t)\overline{{\hat{X}}}_t+B(t)}{c}dt+\frac{1}{N} \sum_{k=1}^N \sigma_k d W_t^k, \quad
\overline{\hat{X}}_0 =\bar{x}_0^N,\\
& d\bar{x}(t)=-\frac{\bar{P}(t)+2a(t)\bar{x}(t)+B(t)}{c}dt, \quad \bar{x}(0)=\bar{x}_0^N,\\
& \hat{P}_t-\bar{P}(t)=\frac{2\alpha}{c}\int_0^t e^{-\alpha(t-s)} a(s)\big(\overline{\hat{X_s}}-\bar{x}(s)\big)ds.
\end{align*}
Applying Grönwall’s lemma  and \textbf{Assumption} $\boldsymbol{(A_2)}$ yield 
\begin{equation*}
    \mathbb{E}\left[\sup_{0\le t\le T} \left|\bar{x}(t)-\overline{\hat{X}}_t\right|^2\right]\le \mathbb{E}\left[O(1)\frac{1}{N^2}\left(\sqrt{\sum_{k=1}^N \sigma_k^2}\right)^2\right]=O\left(\frac{1}{N}\right),
\end{equation*}
which further implies, using the Cauchy–Schwarz inequality,
\begin{align*}
    \|\hat{P}-\bar{P}\|^2&\le O(1) \mathbb{E}\left[\int_0^T \bigg( \int_0^t\left(e^{-\alpha(t-s)} \right)^2ds\bigg)\bigg ( \int_0^t \left(\overline{\hat{X}}_s-\bar{x}(s)\right)^2ds\bigg)dt\right]\\
    &\le O(1) \left(\int_0^T \bigg( \int_0^t\left(e^{-\alpha(t-s)} \right)^2ds\bigg)dt\right)  \mathbb{E}\left[ \int_0^T \left(\overline{\hat{X}}_s-\bar{x}(s)\right)^2ds\right]\\
    &=O\left(\frac{1}{N}\right).
\end{align*}
Then, from \eqref{price process},
\begin{equation*}
    P_t-\hat{P}_t=-\frac{\alpha}{N}\int_0^t e^{-\alpha(t-s)} (v_s^i-\hat{v}_s^i)ds.
\end{equation*}
From the price process \eqref{price process}, we deduce by another application of the Cauchy–Schwarz inequality that
\begin{equation*}
    \|P-\hat{P}\|^2\le \frac{1}{N^2 }\left(\int_0^T dt\int_0^t ( e^{-\alpha(t-s)}\alpha )^2ds\right)\mathbb{E}\left[ \int_0^T (v_s^i-\hat{v}_s^i)^2ds\right]=
O\bigg(\frac{1}{N^2}\bigg),
\end{equation*}
where the last equality follows from the fact that $\|\hat{v}^i\|$ and $\|v^i\|$ are all bounded by $O(1)$ numbers. These estimates imply  
\begin{equation*}
    \|\hat{P}-\bar{P}\|=O\left(\frac{1}{\sqrt{N}}\right),\quad
   \|P-\hat{P}\|= O\bigg(\frac{1}{N}\bigg),  \quad  \|P-\bar{P}\|\le\|P-\hat{P}\|+ \|\hat{P}-\bar{P}\|= O\left(\frac{1}{\sqrt{N}}\right),
\end{equation*}
where the last step applies  the triangle inequality.
Applying the Cauchy-Schwartz inequality once more gives
\begin{equation*}
    |I_2|\le  \|v^i\|\cdot\|P-\bar{P}\|=O\bigg(\frac{1}{\sqrt{N}}\bigg),
\end{equation*}
\begin{equation*}
    |I_3|\le  \|\hat{v}^i\|\cdot\|\hat{P}-\bar{P}\|=O\bigg(\frac{1}{\sqrt{N}}\bigg),
    \end{equation*}
which establishes \eqref{ooo}.
Finally, combining all the above arguments, we conclude that
\begin{equation*}
    J_i(v^i, \hat{\boldsymbol{v}}^{-i})- J_i(\hat{v}^i, \hat{\boldsymbol{v}}^{-i})\ge O\bigg(\frac{1}{\sqrt{N}}\bigg).
\end{equation*}
Hence, the proof is complete.
\end{proof}

\section{The Cooperate Games}

In this section, we aim to identify a set of strategies that achieve an $\epsilon$-optimal expected social cost. Since our objective is to minimize $J_{soc}$, it is essential to quantify the influence of agent 
$i$'s control on the value of
 $J_{soc}$. To this end, we decompose the price process 
$P$ into components that are affected and unaffected by agent $i$,  and then analyze how $v^i$ contribute to  $J_{soc}$.

 Our approach can be outlined in
the following steps.
\begin{enumerate}
    \item 
    We first show that under the optimal strategies of $J_{soc}$, the control of each agent $i$' must minimize the following auxiliary control problem:
\begin{align}\label{optimal property for auxilary Obj} 
 \check{J}_i(v^i)&:= \mathbb{E}\left[ \int_{0}^{T} \bigg(\frac{\eta}{2}  (X_t^i-\kappa)^2+\frac{c}{2} (v_t^i)^2+v_t^i P_t - \frac{u_t^i }{N} \sum_{k=1,k\neq i}^N\check{v}_t^k\bigg)dt  +\frac{\gamma}{2}  (X_T^i-\zeta)^2\right],
\end{align}
subject to 
\begin{equation}\label{eq:socPXi}
\left\{
\begin{aligned}
 &\frac{d P_t}{dt}=\alpha\bigg(\beta-\frac{1}{N}\sum_{k\neq i}^N\check{v}_t^k -\frac{1}{N} v_t^i-P_t\bigg),\quad P_0=p_0,\\
 & dX_t^i=v_t^idt+\sigma_idW_t^i,\quad  X_0^i=x_0^i,\\
 &u_t^i=\int_0^t e^{-\alpha(t-s)}\alpha v_s^i ds .
 \end{aligned}
 \right.
\end{equation}
    \item  We approximate the empirical mean $\frac{1}{N}\sum_{k\neq i}\check{v}_t^k$ with a continuous function $(\bar{q}(t))_{0\le t\le T}$. Subsequently, we construct a continuous approximation function  $(\bar{p}(t))_{0\le t\le T}$ to the price process $(P_t)_{t\in[0,T]}$, which evolves according to
    \begin{equation}\label{approximate price process 2}
        d \bar{p}(t)=\alpha(\beta-\bar{p}(t)-\bar{q}(t))dt,\quad \bar{p}(0)=p_0.
    \end{equation} 
    
    \item For each agent $i=1,\ldots,N$, applying the above approximation yields  the following  auxiliary objective function, 
    \begin{equation}\label{Obj for Social HJB}
        J'_i(v^i;\bar{p},\bar{q})=\mathbb{E}\left[ \int_{0}^{T} \Big(\frac{\eta}{2}  (X_t^i-\kappa)^2+\frac{c}{2} (v_t^i)^2 +v_t^i\bar{p}(t) -\bar{q}(t) u_t^i \Big)dt  +\frac{\gamma}{2}  (X_T^i-\zeta)^2 \right],
    \end{equation}
    subject to the state dynamics $(X^i, u^i) $ given by  \eqref{eq:socPXi}.
    \item By solving  the  control problem  $J'_i$, we obtain the corresponding optimal strategies and derive the consistency condition that the function $\bar{q}$ must satisfy. Finally, we verify that the strategies generated from the solution of this consistency condition achieve asymptotic social optimality.
\end{enumerate}
Our first step is to evaluate precisely the impact of   a single agent's strategy on  the expected social cost  $J_{soc}$. To do this, we divide $J_{soc}$ into three parts, one of them depends only on agent $i$'s strategy, while the others are independent of it. Specifically, we state that:
\begin{theorem}
     If the strategies $\{\check{v}_t^i\}_{i=1}^N$ optimize $J_{soc}$,  then  for any $i=1,\ldots,N$, $\check{v}_t^i$ is the solution for minimizing control problem \eqref{optimal property for auxilary Obj}-\eqref{eq:socPXi} when other agents' maintain their control $\check{v}^{-i}$.

\end{theorem}

\begin{proof}
For $i=1,\ldots, N $, let $J_{soc}=Y_1^i+Y_2^i$,  where
\begin{align*}
&Y_1^i :=\frac{1}{N}   \sum_{k\neq i} \mathbb{E}\left[ \int_{0}^{T} \left(\frac{\eta}{2}  (X_t^k-\kappa)^2+\frac{c}{2} (v_t^k)^2\right)dt +\frac{\gamma}{2} \left(X_T^k-\zeta\right)^2 \right], \\
&Y_2^i :=\frac{1}{N} \mathbb{E}\left[ \int_{0}^{T} \left(\frac{\eta}{2}  (X_t^i-\kappa)^2+\frac{c}{2} (v_t^i)^2+\Big(\sum_{k=1}^N v_t^k\Big)P_t\right)dt  +\frac{\gamma}{2}  (X_T^i-\zeta) ^2  \right].
\end{align*}
Note that the price process  
$P$ is only partially influenced by strategy $v^i$. Solving (\ref{price process}), we   obtain  
 \begin{equation*}
   P_t=e^{-\alpha t}p_0+\int_0^t  e^{-\alpha (t-s) } \alpha \beta ds-\int_0^t  e^{-\alpha (t-s) } \alpha Q_s ds,
 \end{equation*}
    Using the decomposition $Q_t=\frac{1}{N}v_t^i+\frac{1}{N}\sum_{k\neq i}v_t^k$, we further decompose $Y_2^i=Y_3^i+Y_4^i$, where the components are given by
    \begin{align*}
      Y_3^i &=\frac{1}{N}
        \mathbb{E}\left[  \int_0^T  \left(\sum_{k\neq i} v_t^k\right)   \left(e^{-\alpha t}p_0+\int_0^t  e^{-\alpha (t-s) } \alpha \beta ds-\frac{\alpha}{N}\int_0^t  e^{-\alpha (t-s) }   \left(\sum_{k\neq i}v_s^k\right) ds\right)dt\right], \\   
     Y_4^i &= \frac{1}{N} \mathbb{E}\left[ \int_{0}^{T} \left(\frac{\eta}{2}  (X_t^i-\kappa)^2+\frac{c}{2} {(v_t^i)}^2+v_t^iP_t\right)dt  +\frac{\gamma}{2}  \left(X_T^i-\zeta\right)^2 \right]  \\
  & \quad +\frac{1}{N} \mathbb{E}\left[ \int_0^T\left(\sum_{k\neq i}v_t^k\right) \left(-\frac{\alpha}{N} \int_0^t e^{-\alpha (t-s) } v_s^i ds\right)dt \right].
    \end{align*}
 We thus observe that $J_{soc}=Y_1^i+Y_3^i+Y_4^i$, and both $Y_1^i$ and $  Y_3^i$ are  independent of $v^i$. Since $Y_4^i = \frac{1}{N} \check{J}_i(v^i)$,  we conclude that if $(\check{v}_t^1,...,\check{v}_t^N)$ minimizes $J_{soc}$, then for $i=1,\ldots,N$, the strategy $\check{v}_t^i$ minimizes $\check{J}_i(v^i)$. The proof is complete.
\end{proof}

 

Subsequently, we solve the optimal control problem \eqref{Obj for Social HJB}.
From \eqref{eq:socPXi}, we obtain
\begin{equation*}
     du_t^i=(-\alpha u_t^i +\alpha v_t^i) dt, \quad u_0^i=0.
\end{equation*}
Therefore, we consider the system
\begin{equation}\label{xandu}
\left\{
\begin{aligned}
 & dX_s=v_sds+\sigma_idW_s,\quad  X_t=x,\\
 &du_s=(-\alpha u_s +\alpha v_s) ds,\quad u_t=u,\quad \forall t\le s\le T,
 \end{aligned}
 \right.
\end{equation}
and we define the admissible strategy set for $v$ as \begin{align*}
     \mathcal{C}_i=\Big\{v\ |\  &v \text{ is a  real-value process progressively measurable w.r.t } \mathbb{F},\  \|v\|< \infty,\\
    & \text{and there exists a unique strong solution to \eqref{xandu}}\Big\} ,
\end{align*}
By dynamic programming principle, the value function  
\begin{equation*}
V_i(t,x,u)= \inf_{v^i\in \mathcal{C}_i} \mathbb{E}\left[ \int_{t}^{T} \Big(\frac{\eta}{2} (X_s^i-\kappa)^2+\frac{c}{2} (v_s^i)^2-\bar{q}(s) u_s^i+v_s^i\bar{p}(s)\Big)ds  +\frac{\gamma}{2}  (X_T^i-\zeta)^2 \right]
\end{equation*}
satisfies the following HJB equation 
\begin{equation*}
\frac{\partial V_i}{\partial t}+\inf_{v^i\in \mathcal{C}_i} \left\{
\frac{\partial V_i}{\partial x}v_t^i+\frac{\partial V_i}{\partial u}(\alpha v_t^i-\alpha u)
+\frac{c}{2}(v_t^i)^2+v_t^i\bar{p}(t)
\right\}-\bar{q}(t)u+\frac{\sigma_i^2}{2} \frac{\partial^2 V_i}{\partial^2 x}+\frac{\eta}{2} (x-\kappa)^2=0.
\end{equation*}
with terminal condition $V_i(T,x,u)=\frac{\gamma}{2}(x-\zeta)^2$. We adopt the ansatz 
\begin{equation*}
    V_i(t,x,u)=a_i(t)x^2+b_i(t)x+l_i(t)u+f_i(t),
\end{equation*}
which reduces the HJB equation to
\begin{align*}
& a'_i(t)x^2+b'_i(t)x+f_i'(t)+l'_i(t)u+\inf _{v^i\in \mathcal{A}} \Big\{
(2a_i(t)x+b_i(t))v_t^i+\alpha l_i(t) v_t^i-\alpha u l_i(t)
\\
&\quad +\frac{c}{2}{(v_t^i)}^2+v_t^i\bar{p}(t) 
    \Big\}
    -\bar{q}(t)u+ \sigma_i^2a(t)+\frac{\eta}{2} (x-\kappa)^2=0.
\end{align*}
From this, we obtain the explicit optimal control
\begin{equation}
    \check{v}_t^i=-\frac{2a_i(t)\check{X}_t^i+b_i(t)+\bar{p}(t)+\alpha l_i(t)}{c}.\label{optimal control for Social}
\end{equation}
Substituting \eqref{optimal control for Social} into the HJB equation yields the following ODE system:
\begin{flalign}
    &\qquad\qquad\qquad a'_i(t) + \frac{\eta}{2} - \frac{2a^2_i(t)}{c} = 0, & & a_i(T) = \frac{\gamma}{2}, & \label{Social HJB eq 1} \\
    &\qquad\qquad\qquad b'_i(t) - \eta\kappa - \frac{2a_i(t)(\bar{p}(t)+b_i(t)+\alpha l_i(t))}{c} = 0, & & b_i(T) = -\gamma\zeta, & \label{Social HJB eq 2} \\
    &\qquad\qquad\qquad f'_i(t) + \sigma_i^2a_i(t) + \frac{\eta\kappa^2}{2} - \frac{(b_i(t)+\bar{p}(t)+\alpha l_i(t))^2}{2c} = 0, & & f_i(T) = \frac{\gamma\zeta^2}{2}, & \label{Social HJB eq 3} \\
    &\qquad\qquad\qquad l'_i(t) - \alpha l_i(t) - \bar{q}(t) = 0, & & l_i(T) = 0. & \label{Social HJB eq 4}
\end{flalign}

We assert that (\ref{Social HJB eq 1})–(\ref{Social HJB eq 4}) admit a unique solution on 
$[0,T]$. The proof is nearly identical to that of \autoref{a(t) exists theorem}, so we only provide a brief outline.
First, (\ref{Social HJB eq 1}) has a unique solution on 
$[0,T]$ by \autoref{a(t) exists theorem}.
Next, as in the proof of \autoref{a(t) exists theorem}, both \eqref{Social HJB eq 4} and \eqref{Social HJB eq 2} admit unique solutions on 
$[0,T]$. Finally, uniqueness of the solution to (\ref{Social HJB eq 3}) follows immediately by integrating the equation from $t$ to $T$. Since $a_i(t),\ b_i(t)$ and $l_i(t)$ are identical for all agents, we abbreviate them as $a(t),\ b(t)$ and $l(t)$, respectively. Following the same argument as in \autoref{thm:veri}, one can verify the optimality of \eqref{optimal control for Social}, although such verification is not required for establishing the results in \autoref{social main theorem}.

As in the analysis of the non-cooperative game, when all agents adopt the strategy in (\ref{optimal control for Social}), the empirical mean of $X_t^i$, denoted by $\overline{X}_t$, can be approximated by a function $(\bar{x}(t))_{t\in [0,T]}$, whose dynamics satisfy
\begin{equation*}
    d\bar{x}(t)=-\frac{2a(t)\bar{x}(t)+b(t)+\bar{p}(t)+\alpha l(t)}{c}dt,\quad\bar{x}(0)=\bar{x}_0^N.
\end{equation*}

This implies that the requirement in step 3 is equivalent to $\bar{q}(t)=-\frac{2a(t)\bar{x}(t)+b(t)+\bar{p}(t)+\alpha l(t)}{c}$. Combining this with \eqref{approximate price process 2}, \eqref{Social HJB eq 2}, \eqref{Social HJB eq 4}, we observe that the requirement is equivalent to ensuring that the system 
\begin{equation}\label{mean field process evolution for Social}
    \frac{d}{dt}
\begin{pmatrix}
    \bar{p}(t)\\
    \bar{x}(t)\\
    b(t)\\
    l(t)
\end{pmatrix}
=
\begin{pmatrix}
    -\alpha+\frac{\alpha}{c} & \frac{2\alpha a(t)   }{c}& \frac{\alpha}{c}& 
    \frac{\alpha^2}{c}\\
    -\frac{1}{c}&-\frac{2a(t)}{c }& -\frac{1}{c}&-\frac{\alpha}{c}\\
    \frac{2a(t)}{c }& 0 &\frac{2a(t)}{c }&\frac{2\alpha a(t)}{c }\\
    -\frac{1}{c} &-\frac{2a(t)}{c}&-\frac{1}{c}&\alpha-\frac{\alpha}{c}
\end{pmatrix}
\begin{pmatrix}
    \bar{p}(t)\\
    \bar{x}(t)\\
    b(t)\\
    l(t)
\end{pmatrix}
+
\begin{pmatrix}
    \alpha \beta\\
    0\\
    \eta\kappa\\
    0
\end{pmatrix} ,
\end{equation}
together with the associated initial and terminal conditions,
\begin{equation}\label{mean field process evolution for Social: ini&ter values}
    \begin{pmatrix}
    \bar{p}(0),\ 
    \bar{x}(0),\ 
   b(T),\ 
    l(T)
\end{pmatrix}
^\top
=
\begin{pmatrix}
    p_0,\ 
    \overline{x}_0^N,\ 
    -\gamma\zeta,\ 
    0
\end{pmatrix}
^\top,
\end{equation}
admits a solution on $[0,T]$.

As before, the main challenge lies in the boundary conditions being imposed at different time points.
To overcome this, we first construct two processes whose initial values are both specified at $t=0$, as described below.
\begin{proposition}\label{propo for ensuring social MFG}
Equation systems 
\begin{align}\label{Social auxilary eq 1}
    \frac{d}{dt}
\begin{pmatrix}
    \bar{p}_1(t)\\
    \bar{x}_1(t)\\
    {b}_1(t)\\
    {l}_1(t)
\end{pmatrix}
&=
\begin{pmatrix}
    -\alpha+\frac{\alpha}{c} & \frac{2\alpha a(t)   }{c}& \frac{\alpha}{c}& 
    \frac{\alpha^2}{c}\\
    -\frac{1}{c}&-\frac{2a(t)}{c }& -\frac{1}{c}&-\frac{\alpha}{c}\\
    \frac{2a(t)}{c }& 0 &\frac{2a(t)}{c }&\frac{2\alpha a(t)}{c }\\
    -\frac{1}{c} &-\frac{2a(t)}{c}&-\frac{1}{c}&\alpha-\frac{\alpha}{c}
\end{pmatrix}
\begin{pmatrix}
    \bar{p}_1(t)\\
    \bar{x}_1(t)\\
    {b}_1(t)\\
    {l}_1(t)
\end{pmatrix},
\end{align}
and 

\begin{align}\label{Social auxilary eq 2}
    \frac{d}{dt}
\begin{pmatrix}
    \bar{p}_2(t)\\
    \bar{x}_2(t)\\
    {b}_2(t)\\
    {l}_2(t)
\end{pmatrix}
&=
\begin{pmatrix}
    -\alpha+\frac{\alpha}{c} & \frac{2\alpha a(t)   }{c}& \frac{\alpha}{c}& 
    \frac{\alpha^2}{c}\\
    -\frac{1}{c}&-\frac{2a(t)}{c }& -\frac{1}{c}&-\frac{\alpha}{c}\\
    \frac{2a(t)}{c }& 0 &\frac{2a(t)}{c }&\frac{2\alpha a(t)}{c }\\
    -\frac{1}{c} &-\frac{2a(t)}{c}&-\frac{1}{c}&\alpha-\frac{\alpha}{c}
\end{pmatrix}
\begin{pmatrix}
    \bar{p}_2(t)\\
    \bar{x}_2(t)\\
    {b}_2(t)\\
    {l}_2(t)
\end{pmatrix},
\end{align}
associated with initial conditions
$(\bar{p}_1(0),\ \bar{x}_1(0),\ {b}_1(0),\  {l}_1(0))^\top
= (0,0,1,0)^{\top} $ and 
$(\bar{p}_2(0),\ \bar{x}_2(0),\ {b}_2(0),\  {l}_2(0))^\top
= (0,0,0,1)^{\top} $, 
admit a unique solution, denoted by $(\phi_1^*(t))_{t\in [0,T]}=((\bar{p}_1(t),\ \bar{x}_1(t),\ b_1(t),\ l_1(t))^\top)_{t\in [0,T]}$ and $(\phi_2^*(t))_{t\in [0,T]}$ $=((\bar{p}_2(t),\ \bar{x}_2(t),\ b_2(t),\ l_2(t))^\top)_{t\in [0,T]}$ respectively. 
\end{proposition}

    Next, we show that when the initial values of (\ref{mean field process evolution for Social}) are specified only at $t=0$, the system admits a unique solution.
This implies that the solution to \eqref{mean field process evolution for Social} is uniquely determined by $b(0)$ and $l(0)$, once given $\bar{p}(0)=p_0$ and $\bar{x}(0)=\bar{x}_0^N$.
\begin{proposition}\label{propo for ensuring social solution 2}
    For $\forall b_0,\ l_0\in\mathbb{R}$, the initial values $(\bar{p}(0),\ \bar{x}(0), \ {b}(0),\ {l}(0))^\top=(p_0,\ \bar{x}_0^N,\ b_0,\ l_0)^\top$ gives a unique solution that satisfies (\ref{mean field process evolution for Social}), denoted by
    \begin{equation*}
        (\phi_{b_0,  l_0}^*(t))_{t\in [0,T]}=((\bar{p}_{b_0,l_0}(t),\ \bar{x}_{b_0,l_0}(t),\ b_{b_0,l_0}(t),\ l_{b_0,l_0}(t))^\top)_{t\in [0,T]}.
    \end{equation*}
\end{proposition}

  We now introduce a sufficient condition ensuring the existence and uniqueness of solutions to (\ref{mean field process evolution for Social})–(\ref{mean field process evolution for Social: ini&ter values}), stated in \textbf{Assumption} $\boldsymbol{(A_3)}$.

\vspace{3mm}
\noindent\textbf{Assumption $\boldsymbol{(A_3):}$ } The  systems $((\bar{p}_1(t),\ \bar{x}_1(t),\ b_1(t),\ l_1(t))^\top)_{t\in [0,T]}$ and 
$((\bar{p}_2(t),\ \bar{x}_2(t)$, $\ b_2(t),$
$\ l_2(t))^\top)_{t\in [0,T]}$, which are given by \eqref{Social auxilary eq 1} and  \eqref{Social auxilary eq 2} respectively, 
satisfy that $({b}_1(T),\ {l}_1(T))^\top$, $({b}_2(T),\ {l}_2(T))^\top$ are linearly independent vectors. 

\autoref{A3givesol} below illustrates that \textbf{Assumption $\boldsymbol{(A_3)}$} is sufficient for ensuring the existence of solution to (\ref{mean field process evolution for Social}) and (\ref{mean field process evolution for Social: ini&ter values}). 
\begin{proposition}\label{A3givesol}
    If \textbf{Assumption} $\boldsymbol{(A_3)}$ holds, then the existence and uniqueness of solution to (\ref{mean field process evolution for Social}) and (\ref{mean field process evolution for Social: ini&ter values}) also hold.
\end{proposition}

    
As in other works on mean field games, a coercivity condition is required to prevent the absence of an optimal or $\epsilon$-optimal control. However, this condition need not hold for arbitrary positive parameters $c$ and $\eta$. To ensure coercivity, we impose the following assumption.

\vspace{3mm}
\noindent\textbf{Assumption} $\boldsymbol{(A_4):}$  The parameters $c$ and $\eta$ satisfy that $(c-2)\eta> \alpha^2$.
\begin{remark}
     To illustrate the necessity of this assumption, consider the case in which $v_t^i$ remains being a large constant. Then, the term $\frac{1}{N}\mathbb{E}\left[\int_0^T \sum_{i=1}^N (v_t^i P_t) dt\right]$ can become a large negative number, potentially dominating the positive contributions from $\frac{1}{N}\mathbb{E}\left[\int_0^T \sum_{i=1}^N\left( \frac{c}{2}{(v_t^i)}^2+\frac{\eta}{2}(X_t^i-\kappa)^2\right)dt\right]$ when $c $ and $\eta$ are sufficiently small. Consequently, $J_{soc}$ may approach $-\infty$ as the control norm increases, violating coercivity. Since the coefficient in front of $v_t^i P_t$ in $J_{soc}$ is 1, requiring 
$\frac{c}{2}>1$ and a proper lower bound for $\eta$ is sufficient to guarantee coercivity, as stated in \textbf{Assumption} $\boldsymbol{(A_4)}$. 
\end{remark}

\vskip 3mm
Under these assumptions, we construct the process described in \autoref{social main theorem} and show that it yields a set of 
$\epsilon$-optimal strategies.

\begin{theorem}\label{social main theorem}
Assume that \textbf{Assumption} $\boldsymbol{(A_2)}$, $\boldsymbol{(A_3)}$ and $\boldsymbol{(A_4)}$ hold, and the unique solution for (\ref{mean field process evolution for Social}) and (\ref{mean field process evolution for Social: ini&ter values}) is  $(\bar{p}(t),\  \bar{x}(t),\   b(t),\  l(t))^\top$. We construct the system
\begin{equation}\label{form of optimal control in Social}
\left\{
\begin{aligned}
 &d \check{X}_t^i=\check{v}_t^idt+\sigma_idW_t^i,\quad\check {X}_0^i=x_0^i, \\
 & \check{v}_t^i=-\frac{2a(t)\check{X}_t^i+{b}(t)+\bar{p}(t)+\alpha {l}(t)}{c}, \\ 
 & \check{Q}_t=\frac{1}{N}\sum_{i=1}^N \check{v}_t^i,\\
 & d \check{P}_t=\alpha(\beta-\check{P}_t-\check Q_t)dt,\quad\check{P}_0=p_0. 
\end{aligned}
\right.
\end{equation}
Then,  the set of strategies $\{\check{v}^i\}_{1\le i \le N}$ is a set of $\epsilon$-optimal strategies for $J_{soc}$. Moreover, specific calculation shows that under $\{\check v_t^i\}_{i=1}^N$,  $\epsilon_N=O(\frac{1}{\sqrt{N}})$ in \eqref{Social Obj}.  
\end{theorem}
Suppose there exists another admissible control profile
$\boldsymbol{v}=\{v^i\}_{1\le i\le N}$,  
and let the corresponding processes evolve according to
\begin{equation*}
\left\{
\begin{aligned}
 &d {X_t^i}={v_t^i}dt+\sigma_idW_t^i, \quad X_t^i=x_0^i,\\
 &Q_t=\frac{1}{N}\sum_{i=1}^N {v_t^i}, \\ 
 & d P_t=\alpha(\beta-P_t-Q_t)dt,\quad P_0=p_0. 
\end{aligned}
\right.
\end{equation*}
Our objective is to show that \eqref{Social Obj} holds, or equivalently
\begin{equation}\label{socobj}
    J_{soc}(\boldsymbol{v})-J_{soc}(\check{\boldsymbol{v}})\ge -\epsilon_N,
\end{equation}
where  $\epsilon_N=O(\frac{1}{\sqrt{N}})$.
To establish this, we introduce two lemmas.
 \autoref{coercive condition for social} characterizes the coercivity condition.
To facilitate a clearer proof, we first present  \autoref{PRL for social}.
\begin{lemma}\label{PRL for social}
For any $i=1,2,\ldots,N$ and $v\in \mathcal{A}_i$, we construct the process

\begin{equation*}
  x_t=
\begin{pmatrix}
    \tilde{X}_t\\
    \tilde{P}_t
\end{pmatrix},\quad
    \frac{d}{dt}x_t
=
\begin{pmatrix}
    0 &0\\
    0 &-\alpha
\end{pmatrix}
x_t
+
\begin{pmatrix}
    1\\
    -\alpha
\end{pmatrix}
v_t,\quad x_0=
\begin{pmatrix}
    0\\0
\end{pmatrix}.
\end{equation*}
Then, the following result 
\begin{equation*}
    \mathbb{E}\left[\int_0^T v_t \Big(v_t+\alpha \tilde{X}_t+\tilde{P}_t\Big )dt\right]\ge 0\label{Social first lemma}
\end{equation*}
holds.
\end{lemma}

We now proceed to formulate the coercivity condition rigorously, as stated in \autoref{coercive condition for social}.

\begin{lemma}\label{coercive condition for social}
    Assume that \textbf{Assumptions} $\boldsymbol{(A_2)}$ and $\boldsymbol{(A_4)}$ hold. Then there exists constants $ \epsilon>0,\ \epsilon'>0$,  and $ C\in\mathbb{R}$, all independent of $N$ and $i$, such that 
\begin{equation*}
    J_{soc}(\boldsymbol{v}) 
\ge\frac{1}{N} \sum_{i=1}^N \mathbb{E}\left[\int_0^T ({\epsilon}Q_t^2+\epsilon' (v_t^i)^2)dt\right]+C.
\end{equation*}
\end{lemma}

Having established the necessary preparations, we now prove \autoref{social main theorem}. 
\begin{proof}[Proof of  \autoref{social main theorem}]
Let
\begin{equation*}
    \Delta X_t^i=X_t^i-\check{X}_t^i,\quad \Delta v_t^i=v_t^i-\check{v}_t^i,\quad \Delta P_t=P_t-\check{P}_t,\quad \Delta Q_t=Q_t-\check{Q}_t,\quad \Delta u_t^i=u_t^i-\check{u}_t^i,
\end{equation*}
where $u^i$ and $\check u^i$ are given by
\begin{equation*}
    du_t^i=-\alpha u_t^i dt+\alpha v_t^i dt, \quad u_0^i=0,
\end{equation*}
\begin{equation*}
    d\check{u}_t^i=-\alpha \check{u}_t^i dt+\alpha \check{v}_t^i dt,\quad  \check{u}_0^i=0.
\end{equation*}
Let $(\check{X}_t^i,\ \check{P}_t,\ \check{Q}_t)$ denote the processes induced by the optimal control in \eqref{form of optimal control in Social}.
Under this notation, the difference in the social cost can be written as
\begin{align}
    J_{soc}(\boldsymbol{v})-J_{soc}(\check{\boldsymbol{v}})&=\frac{1}{N}\sum_{i=1}^N \mathbb{E}\left[    \int_{0}^{T} \left(\frac{\eta}{2}  \left(\check{X}_t^i+\Delta X_t^i-\kappa\right)^2-\frac{\eta}{2}  \left(\check{X}_t^i-\kappa\right)^2+\frac{c}{2} \left(\check{v}_t^i+\Delta v_t^i
    \right)^2-\frac{c}{2} ({\check{v}_t^i)}^2\right)dt\right]\nonumber\\
    &\quad+\frac{1}{N}\sum_{i=1}^N\mathbb{E}\left[\int_0^T\left((\check{v}_t^i+\Delta{v_t^i})(\check{P}_t+\Delta{P}_t)-\check{v}_t^i\check{P}_t\right) dt\right] \nonumber\\
    &\quad+\frac{1}{N}\sum_{i=1}^N\mathbb{E}\left[\frac{\gamma}{2}  \left(\check{X}_T^i+\Delta X_T^i-\zeta\right)^2 -\frac{\gamma}{2}  \left(\check{X}_T^i-\zeta\right)^2   \right].
    \label{difference of J equality}
\end{align}
Applying Itô’s formula, we convert the terminal contribution into the integral form,

\begin{align}
        &\mathbb{E}\bigg[\frac{\gamma}{2}  \left(\check{X}_T^i+\Delta X_T^i-\zeta\right)^2 -\frac{\gamma}{2}  \left(\check{X}_T^i-\zeta\right)^2\bigg]\nonumber\\
        &=\mathbb{E}\left[\int_0^T \left(a'(t)( \Delta X_t^i)^2+2a'(t)\Delta X_t^i\check {X}_t^i+      2a(t)\Delta v_t^i\Delta X_t^i\right)dt\right]\nonumber\\
    &\quad+\mathbb{E}\left[\int_0^T\left(2a(t)\Delta X_t^i\check {v}_t^i+2a(t)\Delta v_t^i\check{X}_t^i+{b}'(t)\Delta X_t^i+{b}(t)\Delta v_t^i\right)dt\right].\label{equality for terminal terms}
\end{align}
Substituting this into (\ref{difference of J equality}) gives
\begin{equation*}
    J_{soc}(\boldsymbol{v})-J_{soc}(\check{\boldsymbol{v}}): =I_1+I_2,
\end{equation*}
where $I_1$ and $I_2$ are defined by
\begin{align}
    I_1&=\frac{1}{N}\sum_{i=1}^N  \mathbb{E}\bigg[\int_0^T\left((\Delta{X_t^i})^2\left(\frac{\eta}{2}+a'(t)\right)+\Delta{X_t^i} \check{X}_t^i\left(\eta+2a'(t)-\frac{(2a(t))^2}{c}\right)\right)dt\bigg]\nonumber\\
&\quad+\frac{1}{N}\sum_{i=1}^N\mathbb{E}\left[\int_0^T\left(\Delta{X_t^i}\left(-\eta\kappa+{b}'(t)-\frac{2a(t)(\bar{p}(t)+\alpha{l}(t)+{b}(t))}{c}\right)+\Delta v_t^i \left(-\alpha l(t)\right)\right)dt\right]\nonumber\\
    &\quad+\frac{1}{N}\sum_{i=1}^N\mathbb{E}\left[\int_0^T\left(\frac{c}{2}(\Delta v_t^i)^2+\Delta v_t^i \Delta P_t+\bar{q}(t)\Delta P_t+2a(t)\Delta v_t^i \Delta X_t^i\right)dt\right]\nonumber\\
   & =\frac{1}{N}\sum_{i=1}^N\mathbb{E}\left[\int_0^T\left((\Delta{X_t^i})^2\left(\frac{\eta}{2}+a'(t)\right)-\Delta v_t^i \alpha {l}(t)+\frac{c}{2}(\Delta v_t^i)^2\right)dt\right]\nonumber  \\ 
    &\quad+\frac{1}{N}\sum_{i=1}^N\mathbb{E}\left[\int_0^T\left(\bar{q}(t)\Delta P_t+\Delta v_t^i \Delta P_t+2a(t)\Delta v_t^i \Delta X_t^i\right)dt\right],
\label{equality for I_1}
\end{align}
and 
\begin{equation*}
     I_2=\frac{1}{N}\sum_{i=1}^N \mathbb{E}\left[\int_0^T \left( (\check{v}_t^i-\bar{q}(t))\Delta P_t+\Delta Q_t(\check{P}_t-\bar{p}(t)\right)dt\right].
\end{equation*}

We now provide a more detailed explanation of how we derive $I_1$ and what $I_2$ characterizes. To derive $I_1$, we  approximate $\sum_{i=1}^N\check{v}_t^i\Delta P_t$ by $\sum_{i=1}^N\overline{q}(t)\Delta P_t$ and approximate $\sum_{i=1}^N\check{P}_t\Delta v_t^i$ by $\sum_{i=1}^N\overline{p}(t)\Delta v_t^i$. We then use (\ref{equality for terminal terms}) to rewrite the terminal terms, and apply the optimal control formula \eqref{form of optimal control in Social} to eliminate $\check v_t^i$. This yields 
$I_1$, which represents the leading-order change in $J_{soc}$ caused by deviating from the proposed  strategies.  The second part, $I_2$, captures the approximation error due to replacing $(\check v_t^i,\check P_t)$ with $(\bar q,\bar p)$.

Following exactly the same argument as in the proof of \autoref{Proposed control bounded lemma for Nash}, we obtain the existence of an $O(1)$ constant $K''$ such that
\begin{equation*}
    \mathbb{E}\left[\sup_{0\le t\le T }|\check{X}_t^i|^2\right] 
 \vee \mathbb{E}\left[\sup_{0\le t\le T }|\check{Q}_t|^2\right]\le K'', 
    \ \mathbb{E}\left[\int_0^T ({\check{v}_t^i}) ^2 dt\right] \le T \mathbb{E}\left[\sup_{0\le t\le T }(\check{v}_t^i)^2 dt\right]\le K'', 
  \   i=1,\ldots,N,
\end{equation*}
and 
\begin{equation*}
    |J_{i}(\check{v}^i,\check{\boldsymbol{v}}^{-i})|\le K'', \quad i=1,2,\ldots ,N,
\end{equation*}
which further implies
\begin{equation*}
    J_{soc}(\check{\boldsymbol{v}})\le K''.
\end{equation*}

By the coercivity condition in \autoref{coercive condition for social}, there exists $O(1)$ constant  $M_1>0$ such that, if $\mathbb{E}\left[ \int_0^T Q_t^2dt\right]>M_1$, we have
\begin{equation*}
    J_{soc}(\boldsymbol{v})\ge J_{soc}(\check{\boldsymbol{v}}).
\end{equation*}
This inequality directly gives us the desired result. Thus, it suffices to establish the desired inequality \eqref{socobj} under the regime $\mathbb{E}\left[\int_0^T Q_t^2dt\right]\le M_1$. Indeed,  we will show that 
\begin{equation*}    
        I_1\ge 0,\quad 
        I_2=O\bigg(\frac{1}{\sqrt{N}}\bigg).
\end{equation*}
We first prove the claim for the term $I_1$. The next step relies on the identity
\begin{equation}\label{equation for simplifying I_1}
 \mathbb{E}\left[\sum_{i=1}^N \int_0^T \left( \bar{q}(t) \Delta P_t-\Delta v_t^i \alpha {l}(t)\right)dt\right] =0.
\end{equation}
We sketch the argument.
Applying Itô’s formula and using the evolution of ${l}$ in \eqref{Social HJB eq 4}, we obtain
\begin{align}\label{equation for Delta u and v}
     0&=\mathbb{E}\left[ {l}(T)\Delta u_T^i\right]=\mathbb{E}\left[\int_0^T \left({l}'(t)\Delta u_t^i+{l}(t)(-\alpha \Delta u_t^i+\alpha \Delta v_t^i)\right)dt\right]\nonumber\\
     &=\mathbb{E}\left[\int_0^T \left(\bar{q}(t)\Delta u_t^idt +\alpha{l}(t) \Delta v_t^i\right)dt\right],
\end{align}
with $\Delta{P_t}$ and $\Delta u_t^i$ satisfying,
\begin{equation}
    d \Delta{P_t}=\alpha(-\Delta{P_t}-\Delta Q_t)dt,\quad\Delta{P_0}=0,\label{Delta P process}
\end{equation}
 \begin{equation}
    d\Delta u_t^i=\alpha(- \Delta u_t^i +\Delta v_t^i)dt,\quad \Delta u_0^i=0.\label{Delta u_t^i process}
\end{equation}
Since $\Delta Q_t$ is the average of $\Delta v_t^i$, it follows from these dynamics that
\begin{equation}
    \frac{1}{N} \sum_{i=1}^N \Delta u_t^i=-\Delta P_t,\label{P,u relationship}
\end{equation}
and then
\begin{equation*}
    \mathbb{E}\left[\sum_{i=1}^N \int_0^T \left( \bar{q}(t) \Delta P_t-\Delta v_t^i \alpha {l}(t)\right)dt\right]=\mathbb{E}\left[\sum_{i=1}^N  \int_0^T \left(-\bar{q}(t) \Delta u_t^i-\Delta v_t^i \alpha {l}(t)\right)dt \right]=0,
\end{equation*}
where the first and the second step are derived from (\ref{P,u relationship}) and (\ref{equation for Delta u and v}) respectively. This leads directly to
\eqref{equation for simplifying I_1}. Substituting \eqref{equation for simplifying I_1} into the expression for $I_1$ and applying Itô’s formula again, we obtain
\begin{equation*}
    I_1=\frac{1}{N}\sum_{i=1}^N\mathbb{E}\left[\int_0^T\left(\sum_{i=1}^N (\Delta{X_t^i})^2\left(\frac{\eta}{2}+a'(t)\right)+\frac{c}{2}(\Delta v_t^i)^2 +\Delta v_t^i \Delta P_t+2a(t)\Delta v_t^i \Delta X_t^i\right)dt\right].
\end{equation*}
 Applying Itô’s formula once more, we obtain
\begin{equation*}
    \mathbb{E}\left[a(T)(\Delta X_T^i)^2\right]=\mathbb{E}\left[\int_0^T \left(a'(t) (\Delta X_t^i)^2+2a(t)\Delta X_t^i \Delta v_t^i\right)dt \right].
\end{equation*}
This further simplifies the expression of $I_1$ into
\begin{equation*} I_1=\frac{1}{N}\sum_{i=1}^N\mathbb{E}\left[\int_0^T \left(\frac{\eta}{2}(\Delta{X_t^i})^2+\frac{c}{2}(\Delta v_t^i)^2 +\Delta v_t^i \Delta P_t\right)dt+\frac{\gamma}{2}(\Delta X_T^i)^2\right].
\end{equation*}
Let
\begin{equation*}
    \overline{\Delta X_t}:=\frac{1}{N} \sum_{i=1}^N \Delta X_t^i
\end{equation*}
be the averaged deviation process. 
Note that 
\begin{equation*}
\left\{
    \begin{aligned}
        &d \overline{\Delta X_t}=\Delta Q_tdt,\\
        &d \Delta P_t=-\alpha \Delta P_tdt-\alpha \Delta Q_tdt,\\
        & \overline{\Delta X_0}=0,\quad \Delta P_0=0.
    \end{aligned}
    \right.
\end{equation*}
By  \autoref{PRL for social}, we conclude that
\begin{equation*}
   \mathbb{E}\left[ \int_0^T \left(\Delta Q_t+\Delta P_t+\alpha\overline{\Delta X_t}\right){\Delta Q_t} dt \right]\ge 0.
\end{equation*}
In order to use this inequality, we invoke \textbf{Assumption} $\boldsymbol{(A_4)}$, and apply  fundamental inequality to obtain
\begin{equation*}
    \frac{\eta}{2}\left(\overline{\Delta X_t}\right)^2-\alpha \overline{\Delta X_t}\Delta Q_t+\frac{c-2}{2}  \Delta Q_t^2\ge 0.
\end{equation*}
Subsequently, the convexity of $\frac{\eta}{2}x^2$ leads to
\begin{align}
 I_1 &\ge \frac{1}{N}\sum_{i=1}^N\mathbb{E}\left[\int_0^T \left(\frac{\eta}{2}\left(\overline{\Delta{X_t}}\right)^2+\frac{c}{2}\left(\Delta Q_t\right)^2 +\Delta v_t^i \Delta P_t\right)dt\right]\nonumber\\
 &=  \frac{1}{N}\sum_{i=1}^N\mathbb{E}\left[\int_0^T\left(\Delta Q_t (\Delta P_t+\alpha \overline{\Delta X_t}+\Delta  Q_t) \right)dt\right]\nonumber\\
 &\quad+\frac{1}{N}\sum_{i=1}^N\mathbb{E}\left[\int_0^T\left(\frac{\eta}{2}(\overline{\Delta X_t})^2-\alpha \overline{\Delta X_t}\Delta Q_t+\frac{c-2}{2}  (\Delta Q_t)^2 \right)dt\right] , \label{inequality for I_1}
\end{align}
where  Jensen's inequality is used to replace $\sum_{i=1}^N(\Delta X_t^i)^2$ by $N\overline{\Delta X_t}^2$, and  $\sum_{i=1}^N(\Delta v_t^i)^2$ by $N\Delta Q_t^2$. Since the right hand side of \eqref{inequality for I_1} is nonnegative, we conclude
\begin{equation*}
    I_1\ge 0.
\end{equation*}
The proof for $I_1$ is completed, and it is straightforward to prove the estimate for $I_2$. Similar as in  \autoref{Nash main theorem}'s proof, we have
\begin{equation*}
    \mathbb{E}\left[\int_0^T \bigg(\frac{1}{N}\sum_{i=1}^N\big( \check{v}_t^i-\bar{q}(t)\big)\bigg)^2dt \right]=O\Big(\frac{1}{N}\Big).
\end{equation*}
Using the Cauchy–Schwarz and triangle inequalities, together with the bound  {$\|Q\|^2\le M_1$}, we obtain
\begin{equation*}
    \|\Delta Q\|^2\le O(1) (\|Q\|+\|\check{Q}\|)^2\le O(1).
\end{equation*}
Moreover, solving \eqref{Delta P process} and applying Cauchy–Schwarz again yields
\begin{equation*}
    \Delta P_t=-\alpha\int_0^t e^{-\alpha(t-s)}\Delta Q_sds,
\end{equation*}
and
\begin{equation*}
    |\Delta P_t|^2\le  \alpha^2 \|\Delta Q\|^2\left(\int_0^t e^{-2\alpha(t-s)}ds \right)\le O(1) \|\Delta Q\|^2, \quad\|\Delta P\|^2\le O(1) \|\Delta Q\|^2.
\end{equation*}
Thus,
\begin{equation}\label{I2:1}
     \frac{1}{N}\left|\mathbb{E}\left[\int_0^T  \sum_{i=1}^N\big(\check{v}_t^i-\bar{q}(t)\big)\Delta P_tdt\right]\right|\le \frac{1}{N} \|\Delta P\|\cdot \left \|\sum_{i=1}^N(\check{v}^i-\bar{q})\right \|\le O\bigg(\frac{1}{\sqrt{N}}\bigg).
\end{equation}
Proceeding in the same way as in  \autoref{Nash main theorem}, we get 
\begin{equation*}
    \|\check{P}-\bar{p}\|=O\bigg(\frac{1}{\sqrt{N}}\bigg).
\end{equation*}
Hence,
\begin{equation}\label{I2:2}
    \bigg|\mathbb{E}\left[\int_0^T \Delta Q_t(\check{P}_t-\bar{p}(t))dt\right]\bigg |\le \|\Delta Q\|\cdot \|\check{P}-\bar{p}\|\le O\bigg(\frac{1}{\sqrt{N}}\bigg).
\end{equation}
Therefore, combining \eqref{I2:1} and \eqref{I2:2}, we conclude that
\begin{equation*}
    I_2=O\bigg(\frac{1}{\sqrt{N}}\bigg).
\end{equation*}
The proof is complete.

\end{proof}

\section{Numerical experiments}\label{sec:num}
In this section, we provide a numerical example for our model. We begin by presenting the parameter selection method. We take $[\alpha,\ \beta,\  \eta,\ \kappa,\ \gamma,\ \zeta,\ c,\ p_0,\ T]=[1,\ 4,\ 1,\ 4,\ 2,\ 9,\ 4,\ 3,\ 2]$. 
We assume that $x_i,\ \forall i=1,2,\ldots,N$ and $\sigma_i,\ \forall i=1,2,\ldots,N$ are independently and identically distributed according to
$U(2,2.5)$ and $U(1,1.5)$ respectively, and the total number of agents, denoted by $N$, is equal to 1000. Under this framework, it is easy to verify that \textbf{Assumptions} $\boldsymbol{(A_2)}$ and $\boldsymbol{(A_4)}$ are satisfied. Moreover, we obtain $B_1(T)=2.98\neq0$,
 $(b_1(T), \ l_1(T))=(2.22, -1.40)$ and $(b_2(T), \ l_2(T))=(3.82, 4.01)$. It is easily checked that \textbf{Assumptions} $\boldsymbol{(A_1)}$ and $\boldsymbol{(A_3)}$ also hold.
 
 Then, we compare the discrepancy between the approximation functions and the true values under the non-cooperate case and the cooperate case through numerical experiments. As shown below,  \autoref{fig:1} and \autoref{fig:2} compare the evolution of $\bar{p}$ and $P$, $\bar{q}$ and $Q$ in the non-cooperative game setting, respectively.  \autoref{fig:3} and \autoref{fig:4} present the same comparisons under the cooperative game setting. 
 From the plot, we can see that in the non-cooperate case, both curves in  \autoref{fig:1} exhibit a decreasing trend, and
  both curves in  \autoref{fig:2}  display a $U$-shaped trajectory. However, the cooperative strategies not only lead to a steeper initial decline in the price, but also cause  $Q$ to decrease monotonically. Besides, the cooperative case yields more stable and smoother trajectories, especially for the $Q$-related dynamics. This suggests that cooperation among agents can lead to more efficient and consistent outcomes in price and the average trading rate. 
  Although the shapes of the $P$ and $Q$ curves differ between the figures in both cases, 
 these figures all demonstrate that the mean-field approximation performs very well in both cooperative and non-cooperative settings. This property could be valuable in practical applications like smart grids or large-scale decentralized systems.




\begin{figure}[H]
    \centering
    \begin{subfigure}[t]{0.48\linewidth}
        \centering
        \includegraphics[width=\linewidth]{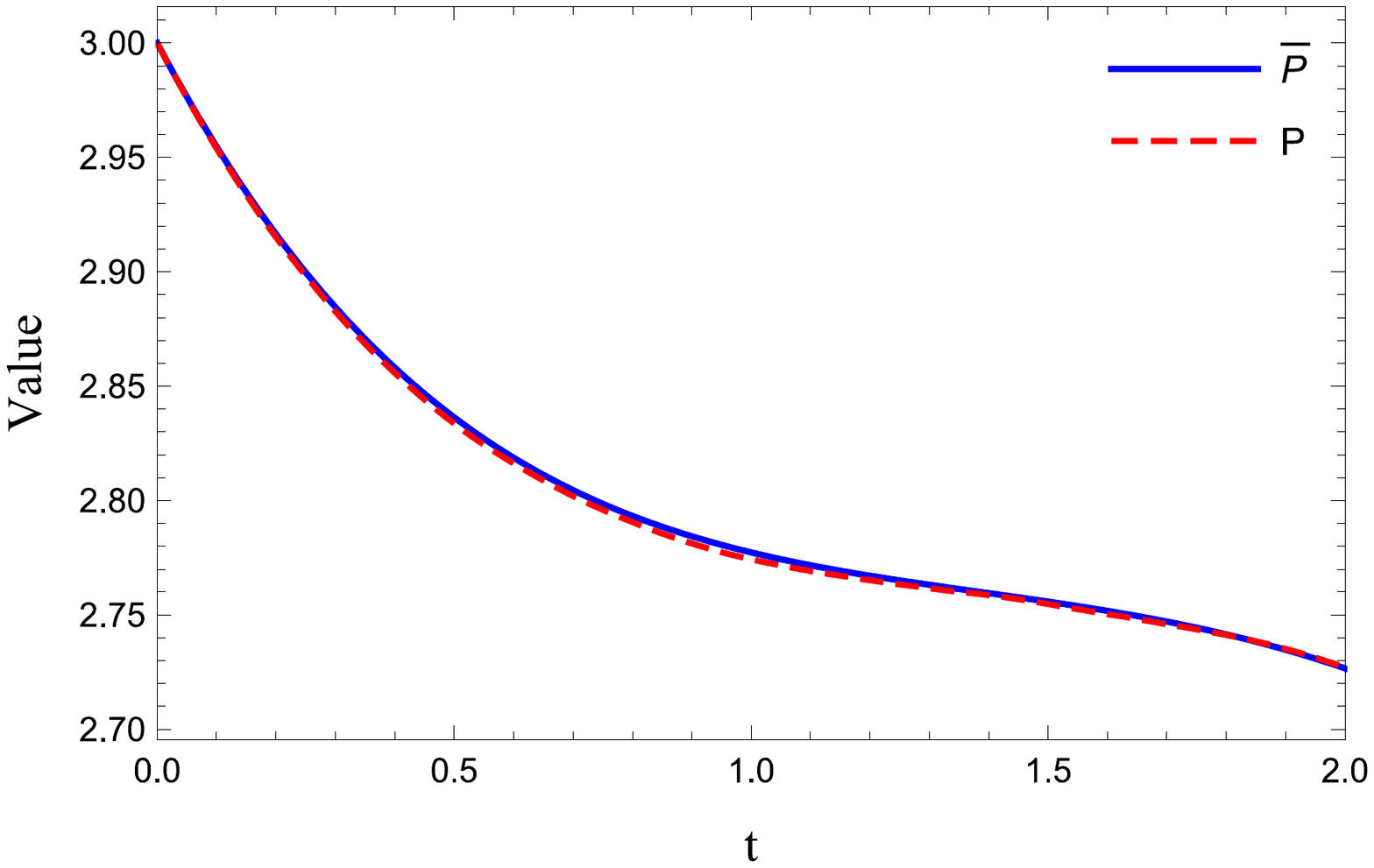}
        \caption{Curves of $\bar{P}$ and $P$ in the non-cooperate case}
        \label{fig:1}
    \end{subfigure}
    \hfill
    \begin{subfigure}[t]{0.48\linewidth}
        \centering
        \includegraphics[width=\linewidth]{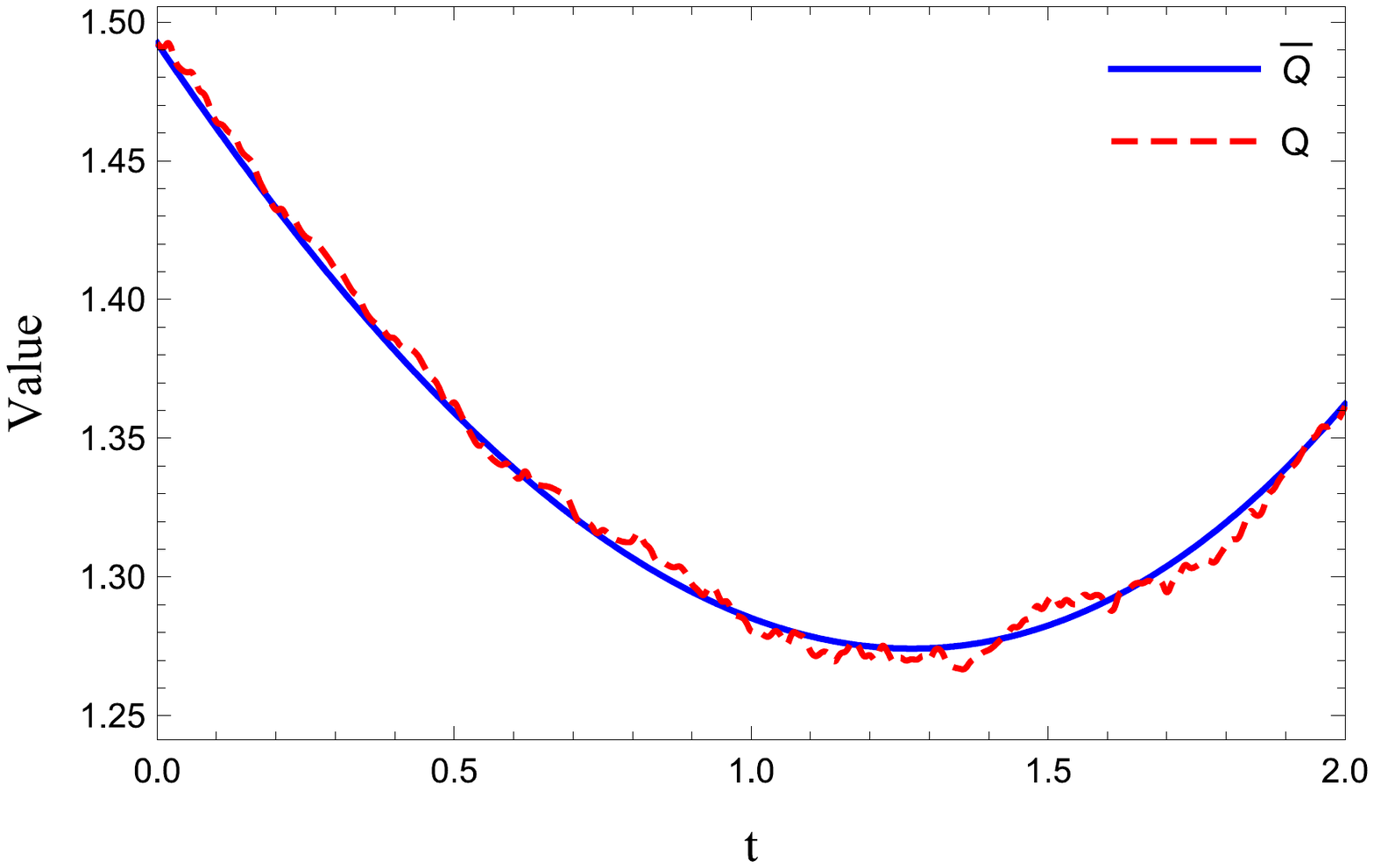}
        \caption{Curves of $\bar{Q}$ and $Q$ in the 
        non-cooperate case}
        \label{fig:2}
    \end{subfigure}
    \caption{Non-cooperate case: evolution of $\bar{P}, P$ and $\bar{Q}, Q$}
    \label{fig:combined}
\end{figure}

\begin{figure}[H]
    \centering
    \begin{subfigure}[t]{0.48\linewidth}
        \centering
        \includegraphics[width=\linewidth]{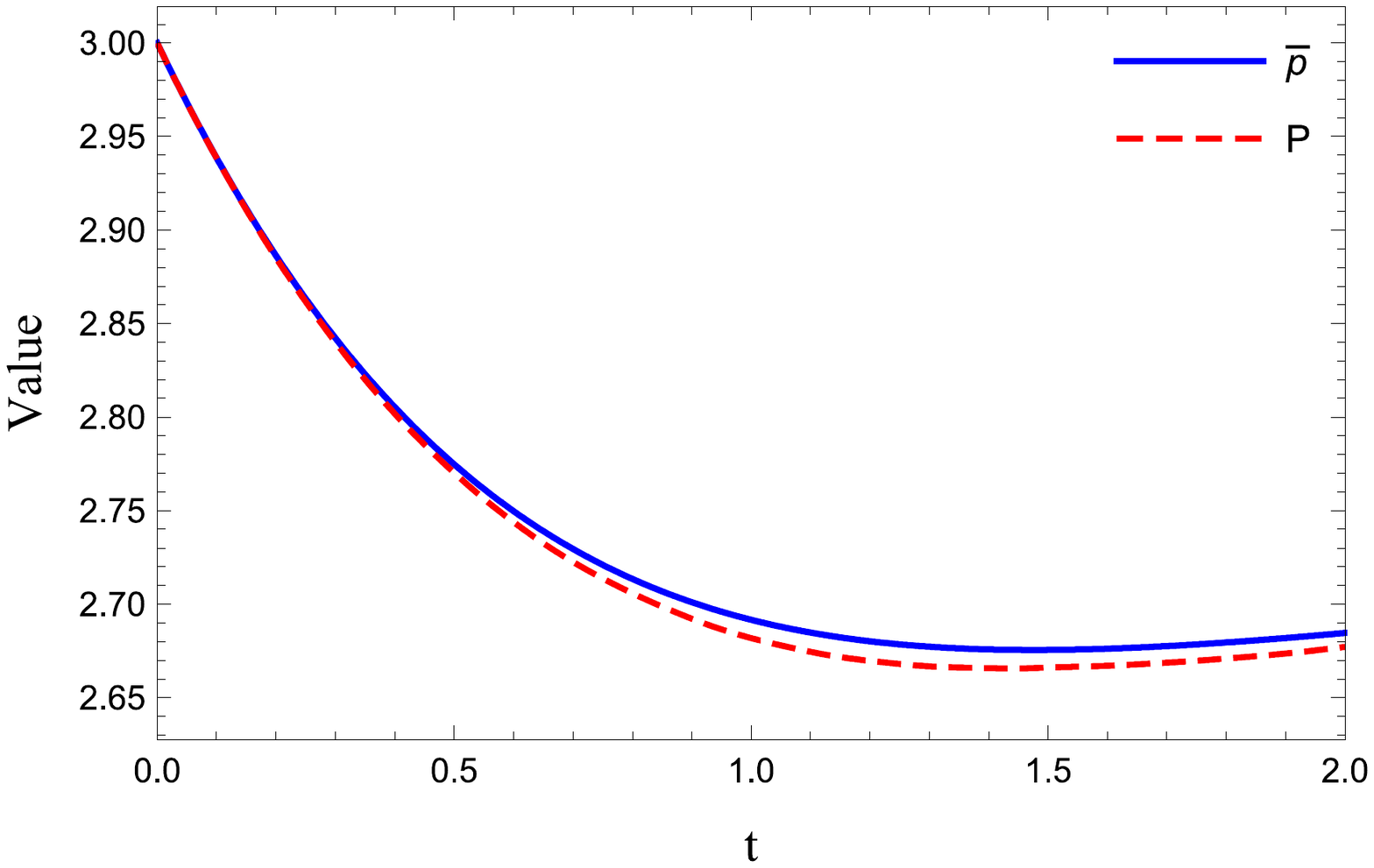}
        \caption{Curves of $\bar{p}$ and $P$ in the cooperate case}
        \label{fig:3}
    \end{subfigure}
    \hfill
    \begin{subfigure}[t]{0.48\linewidth}
        \centering
        \includegraphics[width=\linewidth]{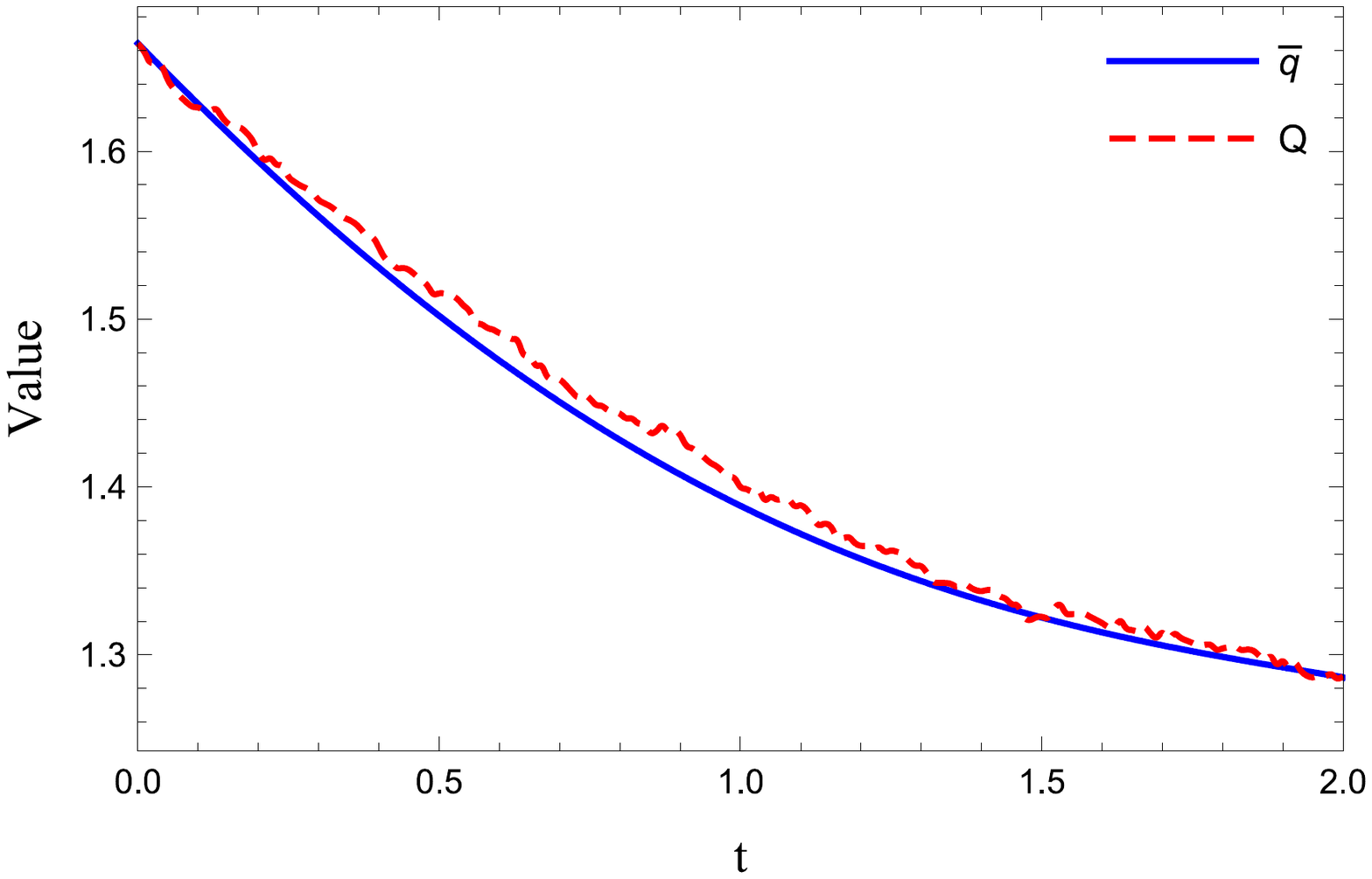}
        \caption{Curves of $\bar{q}$ and $Q$ in the cooperate case}
        \label{fig:4}
    \end{subfigure}
    \caption{Cooperate case: evolution of $\bar{p}, P$ and $\bar{q}, Q$}
    \label{fig:cooperate}
\end{figure}

As shown in  \autoref{vcom}, we also compare the evolution of the control variable 
$v$ for the first 10 agents under both cases. In \autoref{fig:5}, which corresponds to the non-cooperative case, the trajectories of $v^i$
  appear more tightly clustered, with relatively smaller deviations from one another over time. In contrast, \autoref{fig:6}, which shows the cooperative case, exhibits greater divergence in the trajectories of $v^i$. 
Overall, the figures illustrate how cooperation introduces more variability in individual control paths, while non-cooperation leads to more uniform behavior across agents. This reflects the fact that, in a cooperative setting, agents can  adjust their behaviors more flexibly to improve overall system performance.

  \begin{figure}[H]
    \centering
    \begin{subfigure}[t]{0.48\linewidth}
        \centering
        \includegraphics[width=\linewidth]{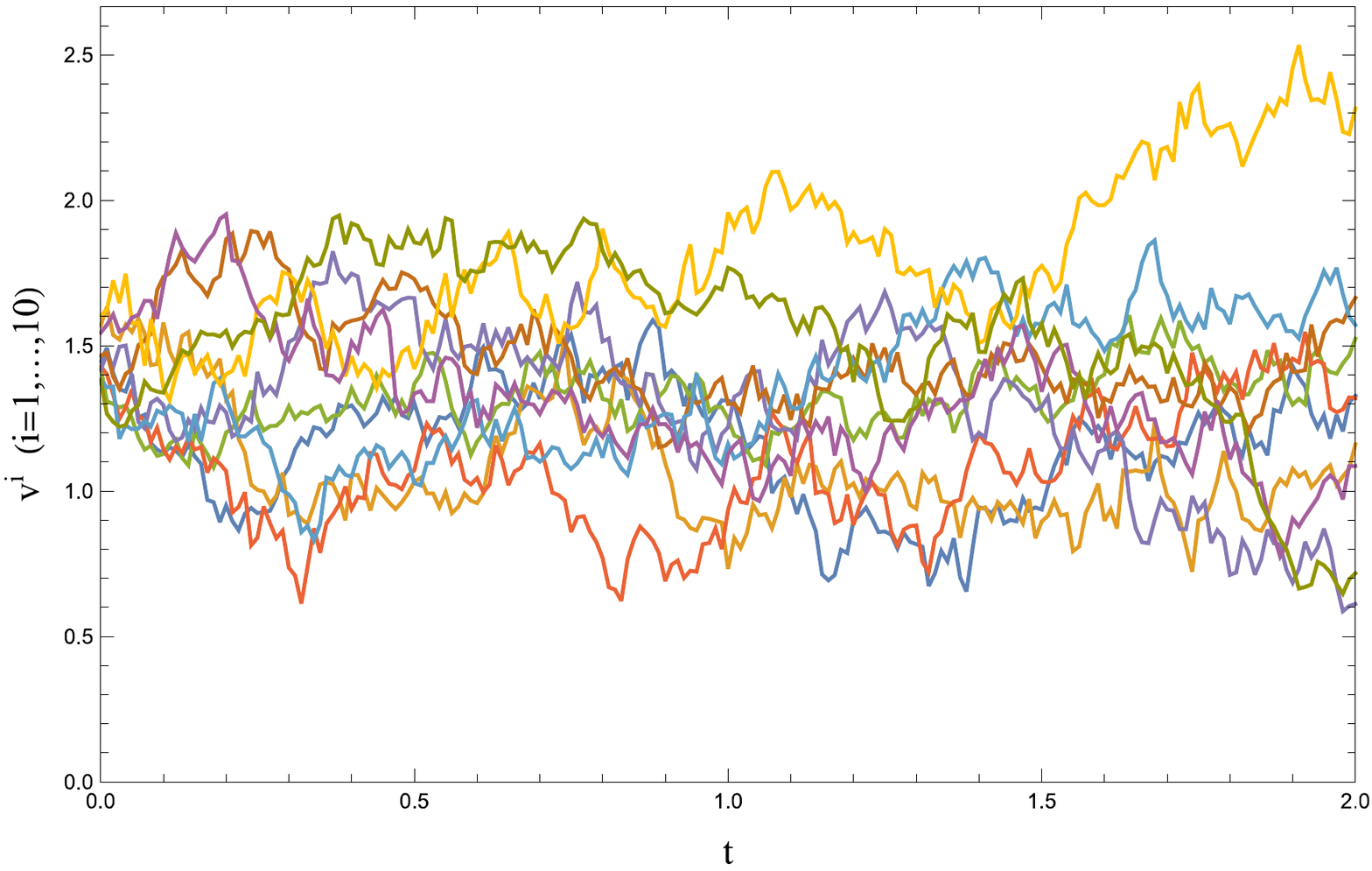}
        \caption{Evolution of $v^i$ for agent 1 to 10 (the non-cooperate case)}
        \label{fig:5}
    \end{subfigure}
    \hfill
    \begin{subfigure}[t]{0.48\linewidth}
        \centering
        \includegraphics[width=\linewidth]{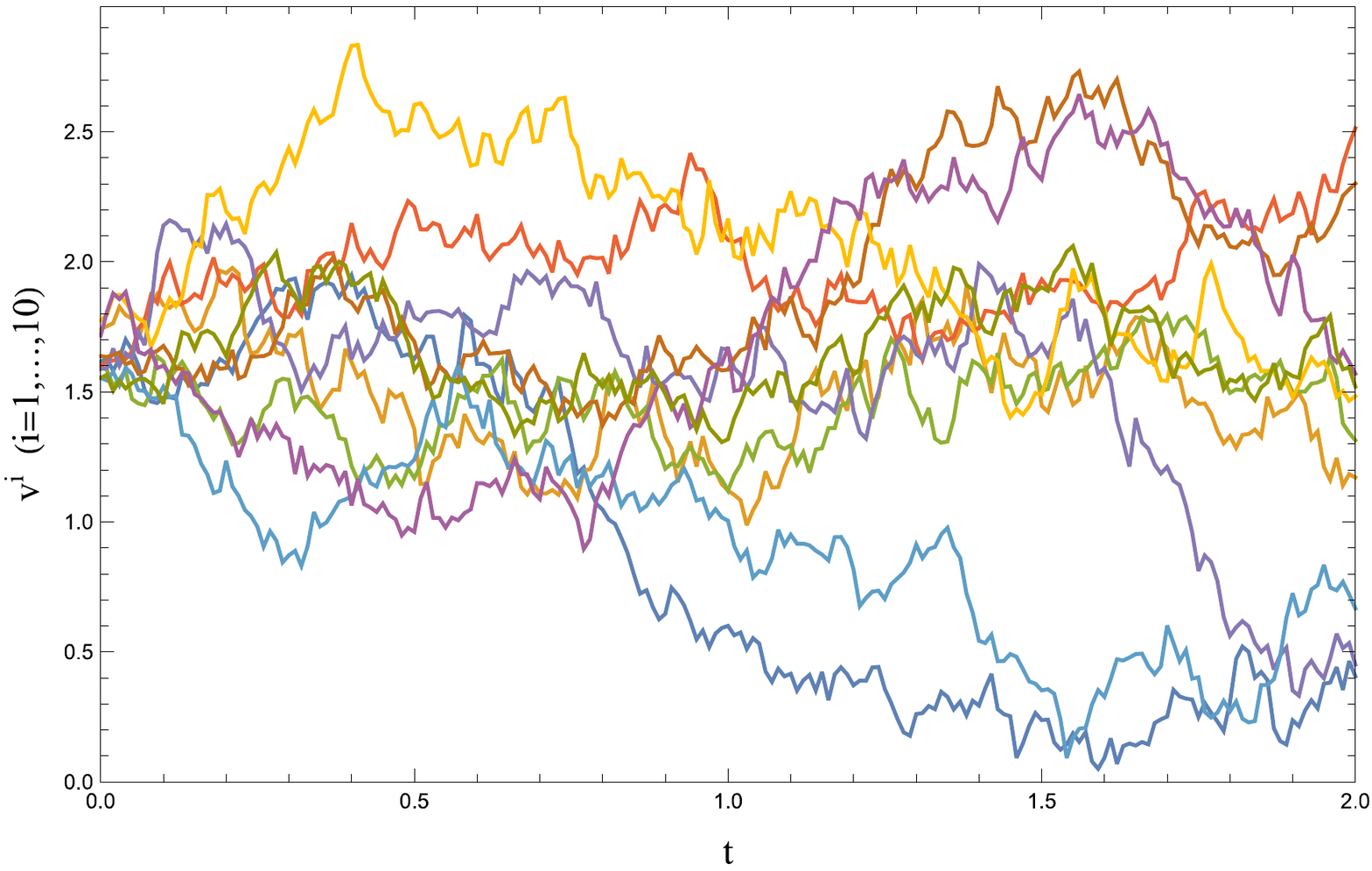}
        \caption{Evolution of $v^i$ for agent 1 to 10 (the cooperate case)}
        \label{fig:6}
    \end{subfigure}
    \caption{Comparison of the evolution process of $v^i$ under both cases}
  \label{vcom}
\end{figure}

\section{Conclusion}\label{sec:con}
This paper has developed a strategic framework for smart grids with many agents using mean field game methodology. Under sticky prices and finite time horizons, we analyze both the cooperative and non-cooperative settings. In the cooperative case, we derive strategies that minimize the expected social cost, while in the non-cooperative case, we construct approximate Nash equilibrium. Numerical experiments demonstrate the effectiveness of these approximations, offering insights for managing large-scale smart grid systems.

\appendix
\section{Proofs of Auxiliary Lemmas}\label{sec:appendix}   	

In this appendix, we collect the proofs of the auxiliary lemmas. We start with the proof of  \autoref{Proposed control bounded lemma for Nash}.
\begin{proof}[Proof of  \autoref{Proposed control bounded lemma for Nash}]
    It follows from \eqref{Nash process} that
    \begin{equation*}
        d \hat{X}_t^i=-\frac{\bar{P}(t)+2a(t)\hat{X}_t^i+B(t)}{c} dt+\sigma_i d W_t^i,\quad \hat{X}_0^i=x_0^i,
    \end{equation*}
which implies that
    \begin{equation*}
        d \overline{\hat{X}}_t=-\frac{\bar{P}(t)+2a(t)\overline{\hat{X}}_t+B(t)}{c} dt+\frac{1}{N}\sum_{i=1}^N \sigma_id W_t^i,\quad \overline{\hat{X}}_0=\bar{x}_0^N.
    \end{equation*}

Gronwall Lemma, along with \textbf{Assumption} $\boldsymbol{(A_2)}$, shows that there exists an $O(1)$ number $K$ such that
\begin{equation*}
    \mathbb{E}\left[\sup_{0\le t\le T }|\hat{X}_t^i|^2\right]\le K,\ \quad   \mathbb{E}\left[\sup_{0\le t\le T }\left|\overline{\hat{X}}_t\right|^2\right]\le K,\ \  \forall i=1,\ldots,N .
\end{equation*}

Therefore, from (\ref{Nash process}), 
there exists an $O(1)$ number (still denoted by $K$) that satisfies
\begin{align}\label{bound for X_t^i and v_t^i}
    &  \mathbb{E}\left[\int_0^T (\hat{v}_t^i)^2 dt\right] \le T \mathbb{E}\left[\sup_{0\le t\le T }(\hat{v}_t^i)^2 dt \right]\le K,\  \forall i=1,\ldots,N ,\ \nonumber\\
    & \mathbb{E}\left[\int_0^T \hat{Q}_t^2 dt\right] \le T\mathbb{E}\left[\sup_{0\le t\le T }\hat{Q}_t^2\right]\le K,\ \forall i=1,\ldots,N.
\end{align}
For any $ i$,  we may treat all other agents as a large population of which consists of $N-1$ agents. Applying the calculations from the previous calculations, we obtain
\begin{equation*}
        \mathbb{E}\left[\sup_{0\le t\le T }\left |\frac{1}{N-1}\sum_{k\neq i} \hat{v}_t^k\right |^2\right]\le K.
    \end{equation*}
This leads to 
\begin{equation*}
        \mathbb{E}\left[\sup_{0\le t\le T }\left|\frac{1}{N}\sum_{k\neq i} \hat{v}_t^k\right|^2\right]\le K.
\end{equation*}
To proceed, we require an exact expression of $\hat{P}$. Solving directly from (\ref{price process}) yields
\begin{equation*}
    \hat{P}_t=e^{-\alpha t}p_0+\int_0^t  e^{-\alpha (t-s) } \alpha \beta ds-\int_0^t  e^{-\alpha (t-s) } \alpha \hat{Q}_s ds.
\end{equation*}
It follows that 
\begin{align*}
    \mathbb{E}\left[\sup_{0\le t\le T }|{\hat{P}_t}|^2\right]&\le  2\mathbb{E}\left[\sup_{0\le t\le T }\left\{\left(e^{-\alpha t}p_0+\int_0^t  e^{-\alpha (t-s) } \alpha \beta ds\right)^2\right\}\right]\\
    &\quad+2 \mathbb{E}\left[\sup_{0\le t\le T }\left\{\left(\int_0^t  e^{-\alpha (t-s) } \alpha \hat{Q}_s ds\right)^2\right\}\right]\\
   & \le O(1)+2(\alpha T)^2 \mathbb{E}\left[\sup_{0\le t\le T }|{\hat{Q}_t}|^2\right]
    \\&\le O(1),
\end{align*}
where we have used the bounds in (\ref{bound for X_t^i and v_t^i}).
This shows that
\begin{equation*}
    \mathbb{E}\left[\int_0^T (\hat{P}_t)^2dt\right]\le T \mathbb{E}\left[\sup_{0\le t\le T }|{\hat{P}_t}|^2\right]\le O(1).
\end{equation*}
Moreover, (\ref{bound for X_t^i and v_t^i}) implies that 
\begin{equation*}
   \left |\mathbb{E}\left[\int_0^T \bigg(\frac{\eta}{2}(\hat{X}_t^i-\kappa)^2+\frac{c}{2}(\hat{v}_t^i)^2\bigg)dt+\frac{\gamma}{2}(\hat{X}_T-\zeta)^2\right]\right |\le O(1).
\end{equation*}
The $O(1)$ bound for $\|\hat{v}\|$ and $\|\hat{P}\|$, together with Cauchy-Schwartz inequality, implies
\begin{equation*}
    \left|\mathbb{E}
    \left[\int_0^T \hat{v}_t^i\hat{P}_tdt
    \right] \right|\le \|\hat{v}^i\|\cdot\|\hat{P}\|\le O(1).
\end{equation*}
This establishes that 
\begin{equation*}
    |J_i(\hat{v}^i,\hat{\boldsymbol{v}}^{-i})|\le O(1),\quad i=1,\ldots,N.
\end{equation*}
Let $K'$ be an $O(1)$ number that exceeds $K$ and all the $O(1)$ numbers above. With this, the proof of \autoref{Proposed control bounded lemma for Nash} is complete.
\end{proof}
    We now turn to the proofs of \autoref{PRL for Nash} and \autoref{Nash coercive lemma}.


\begin{proof}[Proof of  \autoref{PRL for Nash}]
    
Note that $v\in \mathcal{A}_i$ implies
\begin{equation*}
    \int_0^T v_t^2dt<\infty,\ a.s.
\end{equation*}
 Consequently, we restrict attention to events for which the above integral is finite. Under such events, the integral in \autoref{PRL for Nash} exists and is finite.
To present the proof clearly, we introduce the following notation:
\begin{equation*}
A_1=\begin{pmatrix}
    0 &0\\
    0 &{-\alpha}
\end{pmatrix}\preceq O,\quad
B_1=
    \begin{pmatrix}
    1\\
    -\frac{\alpha}{N}
\end{pmatrix},
\quad
    C_1=
\begin{pmatrix}
    {\epsilon}_1^* &1
\end{pmatrix},\quad
D_1={\epsilon}_2^*,
\end{equation*}
\begin{equation*}\label{lemma 2.2.2 notion}
    x_t=\begin{pmatrix}
    X_t^*\\
    P_t^*
\end{pmatrix},\quad
a^*={\epsilon}_1^*>0,\quad b^*=\frac{1}{2\alpha D_1}>0,\quad P_1=P_1^\top=diag{(a^*,b^*)}\succ O.
\end{equation*}
Recalling  \eqref{leq:XPstar}, we define
\begin{equation*}
    y_t:=\epsilon_1^{*}X_t^*+P_t^*+{\epsilon}_2^* v_t =C_1x_t+D_1v_t.
\end{equation*}
Our objective is to apply Positive Real Lemma (Section 2.7.2, \cite{Boyd94}) to the system. To proceed, we denote the LMI for the system by
\begin{equation*}
    L=\begin{pmatrix}
    A_1^\top P_1+P_1A_1  & P_1B_1-C_1^\top\\
    B_1^\top P_1-C_1& -D_1-D_1^\top
\end{pmatrix}
=
\begin{pmatrix}
    0  &0&a^*-\epsilon_1^*\\
    0& -2{\alpha b^*}& -\frac{b^*\alpha}{N}-1\\
    a^*-\epsilon_1^*&-\frac{b^*\alpha}{N}-1&-2D_1
\end{pmatrix}.
\end{equation*}
We investigate the conditions under which the LMI is semi-negative definite. By plugging in the exact value of $a^*$ and $b^*$, we discover that
\begin{equation*}\label{LMI for 2.2.2}
L \preceq O\iff O\preceq\begin{pmatrix}
    0 &0&-a^*+\epsilon_1^*\\
    0& 2{\alpha b^*}& \frac{b^*\alpha}{N}+1\\
    -a^*+\epsilon_1^*&\frac{b^*\alpha}{N}+1&2D_1
\end{pmatrix}  =
\begin{pmatrix}
     0 &0&0\\
    0& \frac{1}{\epsilon_2^*}& \frac{1}{2 \epsilon_2^*N}+1\\
    0&\frac{1}{2 \epsilon_2^*N}+1&2\epsilon_2^*
\end{pmatrix}.
\end{equation*}
Therefore,
\begin{equation*}
    L \preceq O\iff \text{det}
    \begin{pmatrix}
        \frac{1}{\epsilon_2^*}& \frac{1}{2 \epsilon_2^*N}+1\\
    \frac{1}{2 \epsilon_2^*N}+1&2\epsilon_2^*
    \end{pmatrix}
    \ge 0 \iff
    N\ge \frac{\sqrt{2}+1}{2{\epsilon_2}^*}.
\end{equation*}
Since we assume $N\ge\frac{\sqrt{2}+1}{2{\epsilon_2}^*}$,  the LMI is semi-negative definite. Positive Real Lemma implies
\begin{equation*}
    \int_0^T v_t (C_1x_t+D_1v_t)dt\ge 0.
\end{equation*}
Substituting the explicit forms of $C_1$ and $D_1$ and taking the expectation, we immediately obtain that  \autoref{PRL for Nash} holds.\\
\end{proof}

\begin{proof}[Proof of  \autoref{Nash coercive lemma}]
    Fix positive numbers $\epsilon_1, \epsilon_2, \epsilon_3,\epsilon_4$ and $\epsilon_5$  such that
    \begin{equation*}
      \epsilon_2=\frac{c}{8}, \quad 
      \epsilon_3<\frac{\eta}{4}, \quad 
      \epsilon_4+\epsilon_5< \frac{c}{8}, \quad 
      \epsilon_3\epsilon_4=\epsilon_1^2. 
    \end{equation*}
Our first step uses  (\ref{price process}), which gives
\begin{equation*}
P_t=e^{-\alpha t}p_0+\int_0^t  e^{-\alpha (t-s) } \alpha \beta ds
-\frac{\alpha}{N}\int_0^t  e^{-\alpha (t-s) }  \Big(\sum_{k\neq i} \hat{v}_s^k\Big)  ds-\frac{\alpha}{N}\int_0^t  e^{-\alpha (t-s) }  v_s^i ds.
   \end{equation*}
  To evaluate how $P$ changes in response to $v$, we decompose it into two components, one dependent on 
$v$ and the other independent of $v$. As shown below, we set
\begin{align*}
& P'_t :=e^{-\alpha t}p_0+\int_0^t  e^{-\alpha (t-s) } \alpha \beta ds
-\frac{\alpha}{N}\int_0^t  e^{-\alpha (t-s) }  \left(\sum_{k\neq i} \hat{v}_s^k\right)  ds, \\
&  P_t'' := -\frac{\alpha}{N}\int_0^t  e^{-\alpha (t-s) }  v_s^i ds.
\end{align*}
We naturally have $P_t=P_t'+P_t''$. 
The above expression for $P''_t$ is equivalently characterized by the dynamics $dP_t''=-\alpha P_t''-\frac{\alpha}{N}v_t^i$, $P''_0=0$.
Both processes $X_t^i$ and $P_t''$  are  driven by 
$v^i$, and their coefficients match those in \autoref{PRL for Nash}. However, applying the lemma directly requires that both initial conditions be zero, which does not hold in general. Thus, we take a detour and construct the process
\begin{equation*}
    d\dot{X}_t^i=v_t^idt,\ \dot{X}_0^i=0.
\end{equation*}
Note that the dynamics and initial value of $(\dot{X}_t^i, \  P''_t)^\top$ coincide with those in  \autoref{PRL for Nash}. Therefore,  we substitute $({X}_t^*,\ P_t^*)^\top$ with $(\dot{X}_t^i, \ P''_t)^\top$ in  \autoref{PRL for Nash}, and obtain that if $N>\frac{\sqrt{2}+1}{2{\epsilon_2}}$, then
\begin{equation}\label{Using lemma 2.2.2}
    \mathbb{E}\left[\int_0^T v_t^i (\epsilon_1\dot{X}_t^i+P''_t+\epsilon_2 v_t^i)dt\right]\ge 0.
\end{equation}
Combining $P_t=P_t'+P_t''$
and (\ref{Using lemma 2.2.2}), we have, for all  $ N>\frac{\sqrt{2}+1}{2{\epsilon_2}}$,
\begin{align}
    \label{Inequality for J_i 1}
    J_i(v^i,\hat{\boldsymbol{v}}^{-i}) & = \mathbb{E}\left[\int_0^T\left(\frac{\eta}{2}\left(X_t^i-\kappa\right)^2+\frac{c}{2}(v_t^i)^2+P_tv_t^i\right)dt+\frac{\gamma}{2}(X_T^i-\zeta)^2\right]\nonumber\\
   & \ge    \mathbb{E}\left[\int_0^T\left(\frac{\eta}{2}\left(X_t^i-\kappa\right)^2+\frac{c}{2}(v_t^i)^2+P_t''v_t^i-|v_t^iP_t'|\right)dt\right]      \nonumber \\
    & =     \mathbb{E}\left[\int_0^T \left(\frac{\eta}{2}\left(X_t^i-\kappa\right)^2+\left(\frac{c}{2}-\epsilon_2-\epsilon_5\right)(v_t^i)^2\right)dt\right]   +\mathbb{E}\left[\int_0^T  v_t^i\left(\epsilon_1\dot{X}_t^i+P''_t+\epsilon_2 v_t^i\right)dt\right]        \nonumber\\
    &\quad  +\mathbb{E}\left[\int_0^T\left(-\epsilon_1 \dot{X}_t^i v_t^i+\epsilon_5 (v_t^i)^2-|P_t' v_t^i |\right)dt\right]\nonumber\\
    & \ge  \mathbb{E}\left[\int_0^T \left(\frac{\eta}{2}\left(X_t^i-\kappa\right)^2+\left(\frac{c}{2}-\epsilon_2-\epsilon_5\right)(v_t^i)^2\right)dt\right]\nonumber\\
    & \quad +\mathbb{E}\left[\int_0^T\left(-\epsilon_1 \dot{X}_t^i v_t^i+\epsilon_5 (v_t^i)^2-|P_t' v_t^i |\right)dt\right].
\end{align}

From \autoref{Proposed control bounded lemma for Nash}, direct calculations and the Cauchy–Schwarz inequality yield
\begin{align*}
    \mathbb{E}\left[  \sup_{t\in [0,T]}(P'_t)^2\right]&\le 2\mathbb{E}\left[\sup_{t\in [0,T]}\left\{\left(e^{-\alpha t}p_0+\int_0^t  e^{-\alpha (t-s) } \alpha \beta ds\right)^2\right\}\right]\\
    &\quad +2\mathbb{E}\left[\sup_{t\in [0,T]}\left\{\left(\frac{\alpha}{N}\int_0^t  e^{-\alpha (t-s) }  \bigg(\sum_{k\neq i} \hat{v}_s^k\bigg)  ds\right)^2\right\}\right]\\
    &\le O(1),
\end{align*}
and hence
\begin{equation*}
    \mathbb{E}\left[\int_0^T (P'_t)^2dt\right]\le T\mathbb{E}\left[  \sup_{t\in [0,T]}(P'_t)^2\right]\le O(1).
\end{equation*}
  The fundamental inequality gives $|P'_tv_t^i|\le \frac{\epsilon_5}{2}(v_t^i)^2+\frac{1}{2\epsilon_5}(P'_t)^2$, leading to
\begin{equation*}
    \mathbb{E}\left[\int_0^T\big( \epsilon _5(v_t^i)^2-|P'_tv_t^i|\big)dt\right]\ge \mathbb{E}\left[\int_0^T(\epsilon _5(v_t^i)^2-\frac{1}{2}\epsilon _5(v_t^i)^2-O(1)(P'_t)^2)dt\right]\ge O(1).\label{Auxilary Inequality}
\end{equation*}
This simplifies (\ref{Inequality for J_i 1}) into the following inequality, which holds once $ N$ is larger than $\frac{\sqrt{2}+1}{2{\epsilon_2}}$,
\begin{equation*}
    J_i(v^i,\hat{\boldsymbol{v}}^{-i})\ge \mathbb{E}\left[\int_0^T \bigg(\frac{\eta}{2}\left(X_t^i-\kappa\right)^2+\left(\frac{c}{2}-\epsilon_2-\epsilon_5\right)(v_t^i)^2-\epsilon_1 \dot{X}_t^i v_t^i\bigg)dt\right]+O(1) \label{Inequality for J_i 2}.
\end{equation*}
We now control the possible negative cross-term $-\epsilon_1 \dot{X}_t^i v_t^i$ using suitable multiples of $(\dot{X}_t^i)^2$
  and 
$(v_t^i)^2$. Since the inequality contains  $X_t^i$, we rewrite $\mathbb{E}[(X_t^i-\kappa)^2]$ in terms of $\dot{X}_t^i$.  From $(X_t^i-\kappa)-(x_0^i+\sigma_i W_t^i-\kappa)=\dot{X}_t^i$, we get
\begin{align*}
    \mathbb{E}[(\dot{X}_t^i)^2]&\le 2\mathbb{E}[(X_t^i-\kappa)^2]+2\mathbb{E}[(x_0^i+\sigma_i W_t^i-\kappa)^2]\\
   &\le 2\mathbb{E}[(X_t^i-\kappa)]^2+4 (x_0^i-\kappa)^2+4\sigma_i^2T,
\end{align*}
and hence
    \begin{align*}
        \mathbb{E}\left[(X_t^i-\kappa)^2\right]
         \ge& \frac{1}{2}\mathbb{E}\left[(\dot{X}_t^i)^2\right]+O(1),
    \end{align*}
   and, for $\forall N>\frac{\sqrt{2}+1}{2{\epsilon_2}}=\frac{4(\sqrt{2}+1)}{c}$
\begin{align*}
    J_i(v^i,\hat{\boldsymbol{v}}^{-i})&\ge \mathbb{E}\left[\int_0^T \bigg(\frac{\eta}{2}\left(X_t^i-\kappa\right)^2+\left(\frac{c}{2}-\epsilon_2-\epsilon_5\right)(v_t^i)^2-\epsilon_1 \dot{X}_t^i v_t^i\bigg)dt\right]+O(1)\\
    &= \mathbb{E}\Bigg[\int_0^T \bigg(\left(\frac{\eta}{2}-\epsilon_3\right)(X_t^i-\kappa)^2+\left(\frac{c}{2}-\epsilon_2-\epsilon_4-\epsilon_5\right)(v_t^i)^2\bigg)dt\\
    &\quad+\mathbb{E}\left[\int_0^T\left(-\epsilon_1 \dot{X}_t^i v_t^i+\epsilon_4(v_t^i)^2+  \epsilon_3 (X_t^i-\kappa)^2 \right) dt\right]+O(1)\\
    & \ge\mathbb{E}\Bigg[\int_0^T \bigg(\left(\frac{\eta}{2}-\epsilon_3\right)(X_t^i-\kappa)^2+\left(\frac{c}{2}-\epsilon_2-\epsilon_4-\epsilon_5\right)(v_t^i)^2\bigg)dt\\
    &\quad+\mathbb{E}\left[\int_0^T\left(-\epsilon_1 \dot{X}_t^i v_t^i+\epsilon_4(v_t^i)^2+\frac{\epsilon_3}{2}(\dot{X}_t^i)^2\right) dt\right]+O(1),\\
    & \ge\mathbb{E}\Bigg[\int_0^T \bigg(\left(\frac{\eta}{2}-\epsilon_3\right)(X_t^i-\kappa)^2+\left(\frac{c}{2}-\epsilon_2-\epsilon_4-\epsilon_5\right)(v_t^i)^2\bigg)dt+O(1),\
\end{align*}
    where the last step relies on the restriction $\epsilon_1^2=\epsilon_3\epsilon_4$. 
   Furthermore,  applying the restrictions on $\epsilon_1, \epsilon_2, \epsilon_3,\epsilon_4$ and $\epsilon_5$ yield the desired result.  Therefore, the proof is complete.

\end{proof}
The proofs of \autoref{PRL for social} and \autoref{coercive condition for social} are presented below.

\begin{proof}[Proof of  \autoref{PRL for social}]
For $v\in \mathcal{A}_i$, we impose the condition $\int_0^T v_t^2 dt<\infty$, as in the proof of  \autoref{PRL for Nash}.
We introduce the notion
\begin{equation*}
 A_2=\begin{pmatrix}
    0 &0\\
    0 & -\alpha
\end{pmatrix}
\preceq O,\quad
B_2=
\begin{pmatrix}
    1\\
    -\alpha
\end{pmatrix},\quad
C_2=\begin{pmatrix}
    \alpha & 1
\end{pmatrix},\quad
D_2=1,
\end{equation*}
and  the output function 
\begin{equation*}
      \tilde{y}_t:= v_t+\alpha \tilde{X}_t+\tilde{P}_t =C_2 \tilde{x}_t+D_2,
\end{equation*}
which allows us to rewrite the system as
\begin{equation*}
    \frac{d}{dt}  \tilde{x}_t=A_2 \tilde{x}_t+B_2,\quad  \tilde{x}_0= (0,\ 0)^{\top}.
\end{equation*}
We further introduce
\begin{equation*}
   \tilde{a} =\alpha,\quad \tilde{b}=\frac{1}{\alpha},\quad P_2=diag{(\tilde{a},\tilde{b})}\succ O,
\end{equation*}
and define the corresponding LMI by
\begin{equation*}
    L_2=\begin{pmatrix}
    A_2^\top P_2+P_2A_2  & P_2B_2-C_2^\top\\
    B_2^\top P_2-C_2& -D_2-D_2^\top
\end{pmatrix}
=
\begin{pmatrix}
    0 &0&\tilde{a}-\alpha\\
    0& -2\alpha \tilde{b}& -\tilde{b}\alpha-1\\
    \tilde{a}-\alpha&-\tilde{b}\alpha-1&-2
\end{pmatrix}.
\end{equation*}
We state that $L_2\preceq 0$, which is equivalent to
\begin{equation*}
    0\preceq
\begin{pmatrix}
    0 &0&-\tilde{a}+\alpha\\
    0& 2\alpha \tilde{b}& \tilde{b}\alpha+1\\
    -\tilde{a}+\alpha&\tilde{b}\alpha+1&2
\end{pmatrix}
=
\begin{pmatrix}
    0 &0&0\\
    0& 2& 2\\
    0&2&2
\end{pmatrix}.
\end{equation*}
This is immediate. By the Positive Real Lemma, we conclude that
\begin{equation*}
    \int_0^T v_t \tilde{y}_tdt=\int_0^T v_t(v_t+\alpha\tilde{X}_t +\tilde{P}_t) dt \ge 0\label{statement for auxilary process},
\end{equation*}
 and the proof follows upon taking expectations.
\end{proof}

\begin{proof}[Proof of  \autoref{coercive condition for social}]
    To proceed with the proof of \autoref{coercive condition for social}, we make use of \textbf{Assumption} $(\boldsymbol{A_4})$. We fix $\epsilon, \epsilon'>0$ small enough, along with $A\in (0,\frac{\eta}{2})$ and $h\in (0,1)$,  such that
\begin{equation}\label{parameter choosing law}
    2\sqrt{\left(\frac{c}{2}-\epsilon'-1-\epsilon\right)Ah}> \alpha.
\end{equation}
To apply \autoref{PRL for social}, which requires zero initial conditions, we introduce three processes  $P_t^\dagger$, $\overline{X}_t$ and $\overline{X}_t'$, whose dynamics satisfy
\begin{align*}
& d\overline{X}_t=Q_tdt+\frac{1}{N}\sum_{k=1}^N \sigma_k dW_t^k:= Q_tdt+\sigma dB_t, \quad\overline{X}_0=\bar{x}_0^N, \quad\bigg(\sigma= \frac{1}{N} \sqrt{\sum_{k=1}^N \sigma_k^2}\bigg), \\
&  d\overline{X}_t'=Q_tdt ,\quad \overline{X}_0'=0, \\
&  d{P}_t^\dagger=\alpha(-Q_t-P_t^\dagger)dt, \quad P_0^\dagger=0.
\end{align*}
  $\overline{X}_t$, the average of $X_t^i $, contains a Brownian motion term and  does not start from zero. The other two processes start at zero and have no Brownian components. From their definitions, we obtain 
\begin{equation*}
    \overline{X}_t=\bar{x}_0^N+\overline{X}_t'+\sigma B_t,
\end{equation*}
and a direct computation shows that
\begin{equation}\label{auxilary price}
    P_t^\dagger=-\int_0^t  e^{-\alpha(t-s)}\alpha Q_s ds.
\end{equation}
The dynamics and initial value of $(\overline{X}_t',\ P_t^\dagger)^\top $ coincide with those of the process  in  \autoref{PRL for social}.
Therefore,  \autoref{PRL for social} implies that
\begin{equation}\label{usePRL}
    \mathbb{E}\left[\int_0^T Q_t(Q_t+P_t^\dagger+\alpha \overline{X_t}')dt\right]\ge 0.
\end{equation}
Next, recall that
\begin{equation*}
   J_{soc}(\boldsymbol{v}) =\frac{1}{N}\sum_{i=1}^N\mathbb{E}\left[\int_0^T\left(\frac{\eta}{2}\left(X_t^i-\kappa\right)^2+\frac{c}{2}(v_t^i)^2+P_tv_t^i\right)dt+\frac{\gamma}{2}(X_T^i-\zeta)^2\right].
\end{equation*}
Since \eqref{usePRL} involves $P_t^\dagger$ and $\overline{X}_t'$, which do not appear in $J_{soc}$, we aim to express $J_{soc}$ in a
form where $\frac{\eta}{2}(X_t^i-\kappa)^2$ and part of $\frac{c}{2}v_t^i$ are replaced by terms of $P_t^\dagger$, $Q_t$ and $\overline{X}_t'$. Our first step uses  the convexity of $L$, which indicates that  
\begin{equation}\label{Inequality for J 1}
    J_{soc} (\boldsymbol{v})\ge \frac{1}{N}\mathbb{E}\left[\sum_{i=1}^N\int_0^T  \bigg(\frac{\eta}{2}(\overline{X}_t)^2-\eta \kappa\overline{X}_t+\frac{\eta\kappa^2}{2}+\Big(\frac{c}{2}-\epsilon'\Big)Q_t^2 +P_t Q_t+\epsilon' (v_t^i)^2\bigg)dt\right].
\end{equation}
Our second step is based on separating the control-independent part of the price.
We introduce the function
\begin{equation*}
    C(t)=e^{-\alpha t}p_0+\int_0^t  e^{-\alpha (t-s) } \alpha \beta ds, \quad \forall t\in[0,T],
\end{equation*}
which is uniformly bounded on $[0,T]$.  Using the exact form of $P_t$, we have
\begin{align}\label{price equation}
    P_t&=e^{-\alpha t}p_0+\int_0^t  e^{-\alpha (t-s) } \alpha \beta ds-\int_0^t  e^{-\alpha (t-s) } \alpha Q_s ds\nonumber\\
&=P_t^\dagger+C(t).
\end{align}
Our third step concerns the relationship between $\overline{X}_t$ and $\overline{X}_t'$. 
\begin{align*}
\mathbb{E}\left[(\bar{x}_0^N+\overline{X}_t')^2\right]
&=\mathbb{E}\left[(\overline{X}_t-\sigma B_t)^2\right]\\
&\le \mathbb{E}\left[(\overline{X}_t)^2\right]+\left(\frac{1}{h}-1\right)\mathbb{E}\left[(\overline{X}_t)^2\right]+\frac{h}{1-h}\mathbb{E}\left[\sigma^2B_t^2\right]+O(1)\\
&= \frac{1}{h}\mathbb{E}\left[(\overline{X}_t)^2\right]+O(1).
\end{align*}
This implies that
\begin{equation}\label{Inequality for X's average}
    \mathbb{E}[(\overline{X}_t)^2]\ge h \mathbb{E}[(\bar{x}_0^N+\overline{X}_t')^2]+O(1).
\end{equation}
Substituting \eqref{price equation} and \eqref{Inequality for X's average} into \eqref{Inequality for J 1} and rearranging terms yield
\begin{align}\label{jsocin1}
    J_{soc}(\boldsymbol{v})&\ge \frac{1}{N}\mathbb{E}\left[\sum_{i=1}^N \int_0^T \left(\frac{\eta h}{2}\left({(\bar{x}_0^N)}^2+2\overline{X_t}'\bar{x}_0^N\right)-\eta\kappa\overline{X_t}'+\left(\frac{\eta h}{2}-A h\right)\overline{X_t}'^2\right)dt\right]\nonumber\\
    &\quad +\frac{1}{N}\mathbb{E}\left[\sum_{i=1}^N \int_0^T\left(P_t^\dagger Q_t+Q_t^2+\alpha Q_t\overline{X_t}'+Ah \overline{X_t}'^2+\left(\frac{c}{2}-\epsilon'-1-\epsilon\right) Q_t^2-\alpha Q_t\overline{X_t}'\right)dt\right]\nonumber\\
    & \quad   +\frac{1}{N}\mathbb{E}\left[\int_0^T\left(C(t)Q_t
  +\epsilon Q_t^2+\epsilon'{(v_t^i)}^2\right)dt\right] +O(1).
\end{align}
We have shown via \autoref{PRL for social} that  $\mathbb{E}\left[\int_0^T(P_t^\dagger Q_t+Q_t^2+\alpha Q_t\overline{X}_t')dt\right]\ge 0$, and by \eqref{parameter choosing law} and the fundamental inequality,  $(Ah\overline{X}_t'^2+(\frac{c}{2}-\epsilon'-1-\epsilon) Q_t^2-\alpha Q_t\overline{X}_t')$  stays non-negative. Hence we  reduce \eqref{jsocin1} into
\begin{align*}
    J_{soc}(\boldsymbol{v}) & \ge \frac{1}{N}\mathbb{E}\bigg[ \sum_{i=1}^N \int_0^T \bigg(\frac{\eta h}{2}\Big(({\bar{x}_0^N})^2+2\overline{X}_t'\bar{x}_0^N\Big)-\eta\kappa\overline{X}_t' +\left(\frac{\eta h}{2}-Ah\right)\overline{X}_t'^2 \\
    &\qquad\qquad +\left(\frac{\epsilon}{2} Q_t^2-\|C(t)\|_{\infty} |Q_t|\right)+\frac{\epsilon}{2} Q_t^2+\epsilon' (v_t^i)^2\bigg)dt\bigg]+O(1).
\end{align*}
Finally, both $(\frac{\eta}{2}h-Ah)\overline{X}_t'^2+\frac{\eta h}{2}\big(({\bar{x}_0^N})^2+2\overline{X}_t'\bar{x}_0^N\big)-\eta\kappa\overline{X}_t'$ and 
$\frac{\epsilon}{2} Q_t^2-\|C(t)\|_{\infty} |Q_t|$ have  $O(1)$ lower bounds, so
\begin{equation*}
    J_{soc}(\boldsymbol{v}) \ge \frac{1}{N}\mathbb{E}\left[\sum_{i=1}^N \left(\int_0^T \left(\frac{\epsilon}{2} Q_t^2+\epsilon' (v_t^i)^2\right)dt\right)\right]+O(1),
\end{equation*}
and the proof of \autoref{coercive condition for social} is complete.

    \end{proof}
    \section{Proofs of Auxiliary Propositions}\label{sec:appendix2}

In this section, we collect the proofs of the propositions.

\begin{proof}[Proof of \autoref{a(t) exists theorem}]

First, we need to confirm that (\ref{Nash HJB eq1}) admits a unique classical solution on $[0,T]$. As stated before, we abbreviate $a_i$ and $B_i$ as $a$ and $B$, respectively. To prove existence, let function 
\begin{equation*}
    Z(a)=-\frac{\eta}{2}+\frac{2a^2}{c}, \quad a\in \mathbb{R}.
\end{equation*}
Note that if $\gamma\neq \sqrt{c\eta}$, \eqref{Nash HJB eq1} is equivalent to 
\begin{equation*}
    \frac{d a}{Z(a)}=dt,\quad a(T)=\frac{\gamma}{2}.                  \label{ODE for a}
\end{equation*}
By integrating the above equation from $T$ to $t$, and observing that if $\gamma=\sqrt{c\eta}$, the solution must be a constant, we obtain that
\begin{equation*}
    \left\{
\begin{aligned}
 &a(t) = -\frac{\sqrt{c \eta}}{2} \frac{1}{\tanh\left( \sqrt{\frac{\eta}{c}}(t - T) - \operatorname{arctanh} \left( \frac{\sqrt{c \eta} }{\gamma} \right) \right)},\quad \gamma>\sqrt{c\eta},\\
 & a(t)\equiv \frac{\sqrt{c\eta}}{2},\quad \gamma=\sqrt{c\eta},\\
 &a(t)=\frac{\sqrt{c\eta}}{2} \operatorname{tanh}\left(\sqrt{\frac{\eta}{c}}\left(t-T\right)+\operatorname{arctanh}\left(\frac{\gamma}{\sqrt{c\eta}}\right)\right), \quad \gamma<\sqrt{c\eta}.
 \end{aligned}
 \right.
\end{equation*}

From the analysis above, we conclude that (\ref{Nash HJB eq1}) gives a unique solution on $[0,T]$.
 In the following, it is convenient to explicitly compute from (\ref{Nash HJB eq2})
 \begin{equation*}
     B(t)=e^{-A(t)}\bigg(-e^{A(T)}\gamma\zeta-\int_t^T \left(\frac{2a(s)\bar{P}(s)}{c}+\eta 
     \kappa\right)e^{A(s)}ds\bigg),
 \end{equation*}
 where $A(t)$ is a primitive function for $-\frac{2a(t)}{c}$, so (\ref{Nash HJB eq2}) gives a unique solution on $[0,T]$. The solution to (\ref{Nash HJB eq3}) obviously exists and is unique, once we integrate both sides from $T$ to $t$. The proof is complete.
\end{proof}

\begin{proof}[Proof of \autoref{propo for existence}]
   
Let
\begin{equation*}
    \bar{k}(t)
    =\begin{pmatrix}
B(t)\\
\bar{x}(t)\\
\bar{P}(t)
\end{pmatrix},
\quad 
    M(t)=
\begin{pmatrix}
    \frac{2a(t)}{c}&0&\frac{2a(t)}{c}\\
    -\frac{1}{c} &-\frac{2a(t)}{c}&-\frac{1}{c}\\
    \frac{\alpha}{c}&\frac{2\alpha a(t)}{c}&-\alpha+\frac{\alpha}{c}
\end{pmatrix}, \quad 
    K=
\begin{pmatrix}
    \eta \kappa\\
    0\\
    \alpha \beta
\end{pmatrix}.
\end{equation*}
 By Peano theorem, the initial condition gives a local solution. We need to verify that this solution extends to $t=T$ and does not blow up.

We introduce the following notation: the Frobenius norm of any matrix $L$ is denoted by $\|L\|_{\mathrm{F}}$. It is clear that $\|M(t)\|_{\mathrm{F}}$ has a uniform bound for $t\in [0,T]$; we denote this bound by $M^*$, and denote $\|\bar{k}(t)\|_{\mathrm{F}}$ by $n(t)$. Multiply both sides of (\ref{mean field process evolution}) with $2 \bar{k}(t)^\top$, we obtain
\begin{equation*}
    \frac{d}{dt} (n^2(t))=2(\bar{k}(t)^{\mathrm{T}}M(t)\bar{k}(t)+\bar{k}(t)^{\top}K)\le 2n^2(t)M^*+n(t)\|K\|_{\mathrm{F}}.
\end{equation*}
By ODE comparison theorems, $n(t)$ stays below a function $N(t)$ characterized by
\begin{equation*}
    d (N^2(t))=2(M^* N^2(t)+\|K\|_{\mathrm{F}}N(t)),\quad N(0)=n(0)+1.
\end{equation*}
It directly follows that $N(t)$ has a uniform bound on $t\in[0,T]$, so $\bar{k}(t)$ stays finite on $[0,T]$.
Therefore, the existence of the solution holds. Moreover, $M(t)$ is Lipschitz continuous w.r.t $t\in [0,T]$, so  uniqueness naturally follows from the Picard-Lindelöf theorem. The proof is complete.

\end{proof}

\begin{proof}[Proof of  \autoref{propo for existence 2}]
We adopt the notation established in the proofs of  \autoref{propo for existence}.  Deducing in the same way as in the proof of  \autoref{propo for existence}, we obtain the existence property. Recalling that $M(t)$ is Lipschitz continuous, we obtain the uniqueness property.
\end{proof}

\begin{proof}[Proof of  \autoref{A_2 imply existence}]
We adopt the notation established previously in  \autoref{propo for existence}.
    By the definition of $\phi_{b_0}$, as shown in  \autoref{propo for existence} , we discover that for any $b_0\neq 0$, $\frac{\phi_{b_0}-\phi_{0}}{b_0}$  solves (\ref{A_1 process}). Consequently, from the uniqueness property provided by  \autoref{propo for existence 2}, we have
    \begin{equation*}
        \phi_{b_0}(t)-\phi_{0}(t)=b_0B_1(t).
    \end{equation*}
    Set $t=T$. From $B_1(T)\neq 0$, we conclude that there exists a unique $b_0\in \mathbb{R}$ that satisfies
    $\phi_{b_0}(T)=-\zeta\gamma$. Thus, there exists only one solution to (\ref{mean field process evolution}) and (\ref{mean field process evolution:boundary value}). The proof is complete.
\end{proof}

\begin{proof}[Proof of \autoref{propo for ensuring social MFG}]
    Set 
    \begin{equation*}
        G(t)=\begin{pmatrix}
    -\alpha+\frac{\alpha}{c} & \frac{2\alpha a(t)   }{c}& \frac{\alpha}{c}& 
    \frac{\alpha^2}{c}\\
    -\frac{1}{c}&-\frac{2a(t)}{c }& -\frac{1}{c}&-\frac{\alpha}{c}\\
    \frac{2a(t)}{c }& 0 &\frac{2a(t)}{c }&\frac{2\alpha a(t)}{c }\\
    -\frac{1}{c} &-\frac{2a(t)}{c}&-\frac{1}{c}&\alpha-\frac{\alpha}{c}
\end{pmatrix}.
    \end{equation*}
 To prove existence, we observe that $\|G(t)\|_\mathrm{F}$ has a uniform bound on $ [0,T]$. Consequently, by working exactly as in  \autoref{propo for existence}, we obtain the uniqueness property. Note that $G(t)$ is continuously differentiable, so the solution of the system is unique,  from the Picard-Lindelöf theorem.
\end{proof}

\begin{proof}[Proof of \autoref{propo for ensuring social solution 2}]
     The existence holds by the same way as in  \autoref{propo for existence}. Recalling from  \autoref{propo for ensuring social MFG} that $G(t)$ is continuously differentiable everywhere, we conclude that the solution is unique, by the Picard-Lindelöf  theorem.
\end{proof}

\begin{proof}[Proof of \autoref{A3givesol}]
    Observe that , for $\forall b_0,\ l_0\in \mathbb{R}$, $b_0\phi_1^*+l_0\phi_2^*+\phi_{0,0}^*$ solves \eqref{mean field process evolution for Social}, and has boundary values $(\bar{p}(0),\ \bar{x}(0), \ {b}(0),\ {l}(0))^\top=(p_0,\ \bar{x}_0^N,\ b_0,\ l_0)^\top$. Then, uniqueness property in   \autoref{propo for ensuring social solution 2} yields that $\phi_{b_0,l_0}^*=b_0\phi_1^*+l_0\phi_2^*+\phi_{0,0}^*$. Therefore, by setting $t=T$, we get $\phi_{b_0,l_0}^*(T)=b_0\phi_1^*(T)+l_0\phi_2^*(T)+\phi_{0,0}^*(T)$.
    By \textbf{Assumption} $(\boldsymbol{A_3})$, there exists a unique pair of constants $b_0,\ l_0\in \mathbb{R}$ such that the initial values $(\bar{p}(0),\ \bar{x}(0), \ {b}(0),\ {l}(0))^\top=(p_0,\ \bar{x}_0^N,\ b_0,\ l_0)^\top$ gives a solution that satisfies (\ref{mean field process evolution for Social}) and (\ref{mean field process evolution for Social: ini&ter values}). Thus, the solution for (\ref{mean field process evolution for Social}) and (\ref{mean field process evolution for Social: ini&ter values}) exists and is unique. The proof is complete.
\end{proof}

\end{document}